\documentclass[review,hidelinks,onefignum,onetabnum]{siamart220329}

\usepackage{lipsum}
\usepackage{amsfonts}
\usepackage{graphicx}
\usepackage{epstopdf}
\usepackage{algorithmic}
\ifpdf
  \DeclareGraphicsExtensions{.eps,.pdf,.png,.jpg}
\else
  \DeclareGraphicsExtensions{.eps}
\fi

\newsiamremark{remark}{Remark}
\newsiamremark{hypothesis}{Hypothesis}
\crefname{hypothesis}{Hypothesis}{Hypotheses}
\newsiamthm{claim}{Claim}

\headers{The Riemann problem for the foam displacement in porous media }{G. Fritis, P. Paz, L. Lozano and G. Chapiro}

\title{On the Riemann problem for the foam displacement  in porous media with linear adsorption\thanks{Submitted to the editors \today.
\funding{The current work was conducted in association with the R$\&$D project ANP 20715-9, “Modelagem matemática e computacional de injeção de espuma usada em recuperação avançada de petróleo” (UFJF/Shell Brazil/ANP). Shell Brazil funds them in accordance with ANP’s R$\&$D regulations under the Research, Development, and Innovation Investment Commitment. These projects are carried out in partnership with Petrobras. G.C. was supported in part by CNPq Grant 306970/2022-8 and FAPEMIG grant APQ-00405-21.
G.C. and L.L. were supported in part by CNPq grant 405366/2021-3}}}

\author{Giulia C. Fritis \thanks{Laboratory of Applied Mathematics (LAMAP), Federal University of Juiz de Fora, Brazil (\email{giulia.fritis@estudante.ufjf.br}, \email{pavel.sejas.paz@ice.ufjf.br}, \email{luis.guerrero@ice.ufjf.br}, \email{grigori@ice.ufjf.br}).}
\and Pavel S. Paz \footnotemark[2]
\and Luis F. Lozano \footnotemark[2]
\and Grigori Chapiro\footnotemark[2]}

\usepackage{amsopn}

\usepackage{mathrsfs} 
\newtheorem{prop}{Proposition}[section]

\definecolor{darkpastelgreen}{rgb}{0.01, 0.75, 0.24}

\definecolor{tiffanyblue}{rgb}{0.04, 0.73, 0.71} 

\definecolor{ao}{rgb}{0.0, 0.0, 1.0}

\usepackage{caption}
\usepackage{subcaption}

\ifpdf
\hypersetup{
  pdftitle={On the Riemann problem for a model describing the foam displacement with linear adsorption},
  pdfauthor={G. Fritis, P. Paz, L. Lozano and G. Chapiro}
}
\fi

\begin{document}
\nolinenumbers
\maketitle

\begin{abstract}
Motivated by the foam displacement in porous media with linear adsorption, we extended the existing framework for two-phase flow containing an active tracer described by a non-strictly hyperbolic system of conservation laws.
We solved the global Riemann problem by presenting possible wave sequences that composed this solution.
Although the problems are well-posedness for all Riemann data, there is a parameter region where the solution lacks structural stability.
We verified that the model implemented on the most used commercial solver for geoscience, CMG-STARS, describing foam displacement in porous media with adsorption satisfies the hypotheses to apply the developed theory, resulting in structural stability loss for some parameter regions.
\end{abstract}



\section{Introduction}
\label{sec:intro}

In this investigation, we study the non-strictly hyperbolic system of conservation laws given by
\begin{eqnarray}
\label{eq:mass_con_S}   \partial_t S + \partial_xf(S,C) &=& 0,\\
\label{eq:evol_con_C}    \partial_t\left[(S+\mathcal{A})C\right] + \partial_x\left[f(S,C)C\right] &=& 0,
\end{eqnarray}
where $(S,C) \in I\times I$, with $I=[0,1]$, $(x,t) \in \mathbb{R}\times \mathbb{R}^+$, $\mathcal{A}>0$ is a constant and $f: I\times I \rightarrow \mathbb{R}$. This system describes a two-phase flow in a porous medium with an active tracer (diluted in the wetting phase), which can be adsorbed to the surrounding matrix. Typically, $f=f(S,C)$ is the fractional flow function of the wetting phase. Here, $S$ represents the saturation of the wetting phase, $C$ indicates the tracer concentration in the wetting phase, and $\mathcal{A}$ represents the linear adsorption of the tracer fluid. 
In this paper, we construct the weak solution of the Riemann problem given by the system \eqref{eq:mass_con_S}-\eqref{eq:evol_con_C} and initial condition
\begin{equation}
\label{eq:initial_cond}
(S(x,0) ,C(x,0) )= \left\lbrace
\begin{array}{ll}
(S_J,C_J), & \textup{if } x < 0, \\
(S_I,C_I), & \textup{if } x \geq 0, 
\end{array}
\right.
\end{equation}
where the sub-indexes $J$ and $I$ indicate the injection (left) and initial (right) conditions, respectively. 
Observe that, Eqs.~\eqref{eq:mass_con_S} and \eqref{eq:evol_con_C} form a system coupled through the fractional flow function $f$. In the particular case when $C$ is constant, the solution construction follows a Buckley-Leverett type solution \cite{BuckleyLeverett1942}, which involves a shock or a rarefaction wave, or a combination of both.

The model presented in Eqs.~\eqref{eq:mass_con_S}-\eqref{eq:initial_cond} extends the investigation of Isaacson \cite{ELI} and Isaacson \& Temple \cite{isaacson1986analysis} by introducing tracer adsorption to the surrounding matrix (the constant $\mathcal{A}$). A similar model has been studied by Johansen \& Winther \cite{johansen1988} considering Langmuir's adsorption, which does not apply to the linear case as it presents different characteristic family properties.
Both \cite{ELI, isaacson1986analysis}, and \cite{johansen1988} aimed at the polymer dissolved in the aqueous phase displacing the oil phase, while in the present paper, we focus on the foam displacement in porous media saturated with water and gas. This problem appears in industrial applications \cite{Hematpur2018,zeng2016effect}. The systems studied in \cite{Barkve1989,da1992riemann} are similar to \eqref{eq:mass_con_S}-\eqref{eq:evol_con_C}, describing the effects of temperature on the oil displacement. Two transition curves are obtained in \cite{Barkve1989,da1992riemann}, while in \cite{johansen1988,ELI,isaacson1986analysis} appears one transition curve, as in the present work. 
In \cite{tang2019three, Mehrabi2022}, the authors investigate three-phase foam displacement in the presence of oil for a few Riemann problems. In the present work, we classify solutions for all possible Riemann problems as in \cite{Barkve1989,da1992riemann,johansen1988,ELI, isaacson1986analysis}.


The well-posedness (including solution existence, uniqueness, and continuous dependence on parameters) is essential for the model's reliability. The works cited above address the solution existence by constructing the compatible wave sequence. 
The uniqueness of the solutions for Riemann problems is a challenging topic addressed by authors in the last years; see \cite{azevedo2014uniqueness,azevedo1996nonuniqueness} and references therein.
For strictly hyperbolic conservation laws systems, studies successfully showed the uniqueness of Riemann's problem solution, for instance, \cite{dafermos2005hyperbolic,liu1975existence,Smoller1994}. On the other hand, for non-strictly hyperbolic systems, such a theory cannot be applied, and each system needs to be investigated separately. For example, uniqueness's proof was provided in \cite{azevedo2014uniqueness,johansen1988}; examples of the non-uniqueness were presented in \cite{azevedo1996nonuniqueness,Barkve1989,ELI}. Besides the solution's continuous dependency on the parameters, similar models demonstrate a lack of structural stability \cite{thesisFred, schecter1996structurally}.
In the present work, we show the well-posedness of the model and that it lacks structural stability.

CMG-STARS is a widely used geoscience reservoir software designed to model and simulate oil and gas recovery processes. 
CMG-STARS is the unquestioned application standard in thermal and advanced processes reservoir simulation \cite{tunnish2019history}.
CMG-STARS is recognized for its capability to represent both experimental and field results, while it also can model complicated chemical performance \cite{norris2011core}. In this paper, we apply the developed theory to the model implemented in CMG-STARS simulator \cite{Stars,zeng2016effect}

This article is structured as follows. 
Section \ref{sec:Prel_Res} presents preliminary results on the fundamental waves appearing in the solution.
Section \ref{sec:RP_construction} contains the principal results concerning the construction of the Riemann problem's solution. 
Section~\ref{Sec:Uniqueness} discusses the well-posedness of the problem.
Section~\ref{sec:application} applies the developed theory to the model implemented in the CMG-STARS simulator and compares the analytical solutions with direct numerical simulations. Finally, discussions are presented in Section~\ref{sec:conclusions}.

\section{Preliminary results}\label{sec:Prel_Res}

Following the literature \cite{ELI, isaacson1986analysis,johansen1988, temple1982global}, we assume that the flux is described by the real function $f(S,C)$, with S-shape for each fixed value of $C$, see the left panel in Fig.~\ref{fig:f_C}. This assumption is common in many applications and represents realistic physics; however, the S-shape can be less obvious in real-world applications, see the right panel in Fig.~\ref{fig:f_C}. We assume that the real function $f=f(S,C)$ satisfies the following properties:
\begin{enumerate}
	\item[$a$)] The function $f\in \mathscr{C}^2$, $f(0,C)=0$ and $f(1,C)=1$ for every $C \in I$. Also, $\partial_S f(0,C) = \partial_S f(1,C) = 0$, for each $C\in I$. 
	\item[$b$)] For each $C \in I$, $f(S,C)$ is a strictly increasing function of $S$ with a single inflection point.
	\item[$c$)] The derivative of $f$ in $C$ satisfies $\partial_C f(S,C) >0$, for $0<S<1$ and $C \in I$. 
 
	\begin{figure}[h!]
		\centering	    
		\includegraphics[height=5cm]{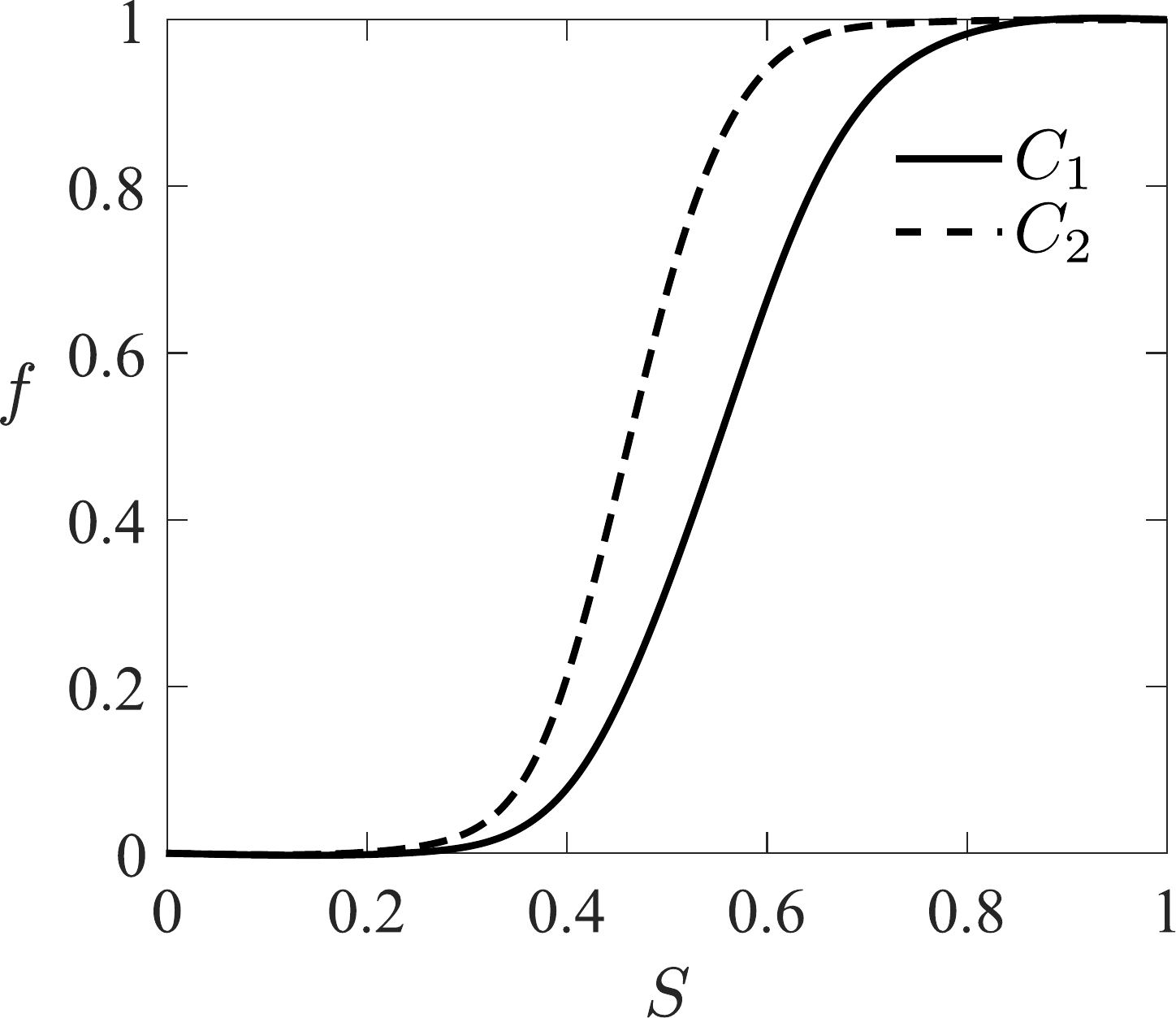} \hspace{0.05\linewidth}
		\includegraphics[height=5cm]{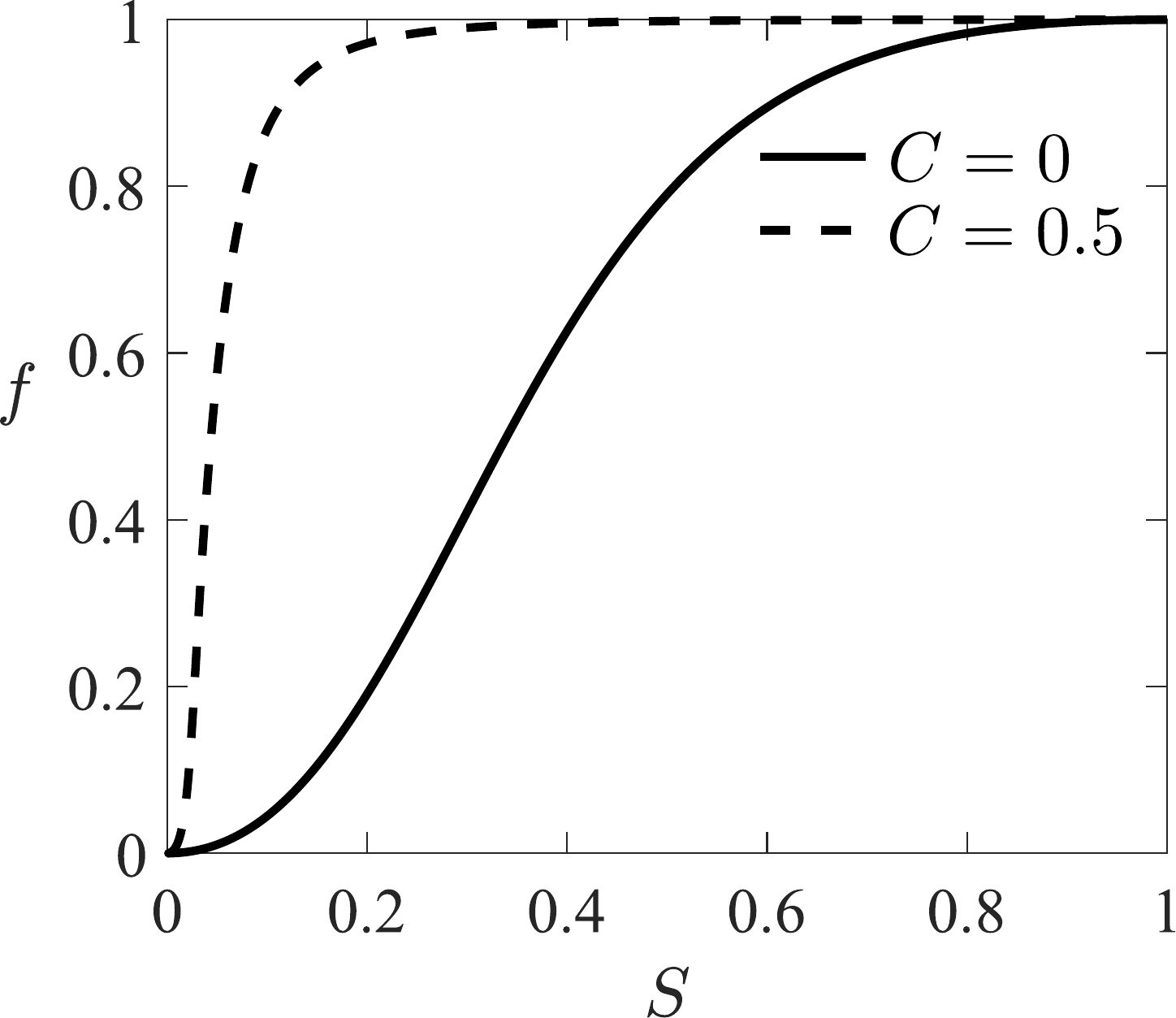}
 		\caption{Fractional flow function for different values of $C$. On the left panel, we present a schematic representation. On the right panel, we use the realistic model described in Subsection~\ref{App:A} with parameter values from Table \ref{Table:Values_parameter}.}
		\label{fig:f_C}
	\end{figure}
\end{enumerate}

The properties $a)$ and $b)$ above are the same as in \cite{ELI, isaacson1986analysis,johansen1988, temple1982global}, while in item $c)$ $f$ is an increasing function of $C$ differently from the same references. The change of variable $\overline{C}=1-C$, results in a similar scenario as  \cite{ELI, isaacson1986analysis,johansen1988, temple1982global}. For this reason, some details in the proofs presented below are omitted. Notice that the system studied here is different from the references above due to the considered linear adsorption (see the constant $\mathcal{A}$ in \eqref{eq:evol_con_C}). This is the main difference in relation to the models studied in referenced works. The modification in item $c)$ corresponds to the fractional flow function $f$ appearing in the local equilibrium model describing the foam displacement implemented in the commercial CMG-STARS simulator \cite{Stars,Zhang2009}, as explained in Subsection~\ref{App:A}.

\subsection{Phase plane $S$-$C$ division}
\label{sec:RP}


The system \eqref{eq:mass_con_S}-\eqref{eq:evol_con_C} can be written in the general form 
\begin{equation}\label{eq:conservative_form}
\partial_t U + A(U) \partial_xU= \mathbf{0},
\end{equation}
where $U$ denotes the vector state $U=(S,C)^T$, $T$ indicates the matrix transpose operator, and $A(U)$ is the $2\times 2$ upper triangular matrix, which is written as
\begin{align}\label{eq:matrix_A}
A(U)=  
\begin{pmatrix} 
\partial_Sf & \partial_Cf \\ 
0 & f /(S+\mathcal{A})  
\end{pmatrix}.
\end{align}
The eigenvalues and eigenvectors associated with the matrix $A$ are given by
\begin{eqnarray}
\label{eq:lambda_C}
\lambda_C &=& f/(S+\mathcal{A}), \quad \quad \,
r_C= \left( \partial_Cf , f/(S+\mathcal{A})-\partial_Sf  \right)^T,  \\
\label{eq:lambda_S}
\lambda_S &=&\partial_Sf,\quad \, \quad \quad \quad \quad r_S=(1,0)^T.
\end{eqnarray}

Our next step is proving that the set of points, where the eigenvalues $\lambda_C$ and $\lambda _S$ coincide, is a curve. This result is analogous to the one presented in \cite{Barkve1989,da1992riemann}.
\begin{prop}\label{pro:Transition_Lemma}
For each fixed $C \in I$, there exists a unique $S^* = S^*(C)$ in the interior of the interval $I$, such that
\begin{equation}\label{Transition_Lemma}
\lambda_C(S^*,C) = \lambda_S(S^*,C).
\end{equation}
\end{prop}
\begin{proof}
Let us consider a fixed $C \in I$ and define $\phi: I \to \mathbb{R}$, such that
$\phi(S) = f(S,C)/(S+\mathcal{A}) - \partial_Sf(S,C)$. In fact, from \eqref{eq:lambda_C}-\eqref{eq:lambda_S}, it follows that $\phi \in \mathscr{C}^1$ and $\phi(S)= 0$ if and only if $\lambda_S(S,C) =\lambda_C(S,C)$
in the interior of the interval $I$. Let us prove that $\phi$ possesses only one root in the open interval $(0,1)$.
Derivation of $\phi$ results in $d_S\phi(S) = - \partial_{SS}f(S,C)(S+\mathcal{A})$. 
Denoting the inflection point of $f(\cdot, C)$ by $S^i = S^i(C)$, yields
\begin{equation*}
d_S\phi(S)< 0  \, \text{ if } S<S^i, \quad
d_S\phi(S)= 0  \, \text{ if } S=S^i, \quad  d_S\phi(S)> 0  \, \text{ if } S>S^i.
\end{equation*}
Therefore, 
$\phi$ has a minimum at $S^i$ and $\phi(S^i)<0$. Once $\phi(0)=0, \ \phi(1) = 1$ and for $S$ greater than $S^i(C)$ we possess $d_S\phi(S)>0$; by the Intermediate Value Theorem there exists a unique $S^*=S^*(C)$ such that $\phi(S^*) = 0$.
For this reason, there exists a unique $S^*$ (see the left panel in Fig.~\ref{fig:diagram}), such that the relation \eqref{Transition_Lemma} is satisfied. Although the S-shape is less evident for the real applications presented in Subsection~\ref{App:A}, the construction shown here is still valid; see the right panel in Fig~\ref{fig:diagram}.
\end{proof}

\begin{figure}[h!]
	\centering	    
	\includegraphics[height=5cm]{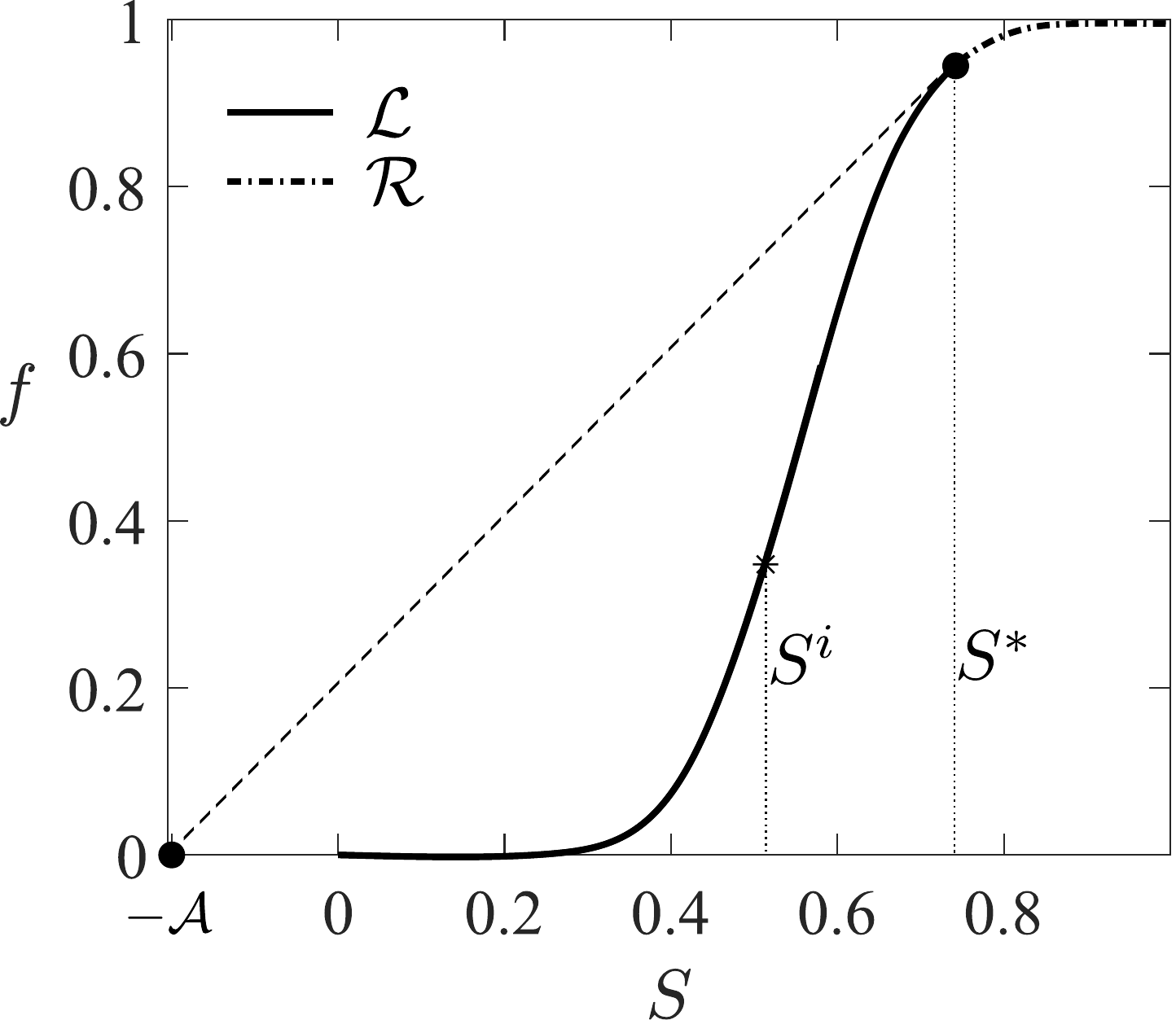} \hfill
	\includegraphics[height=5cm]{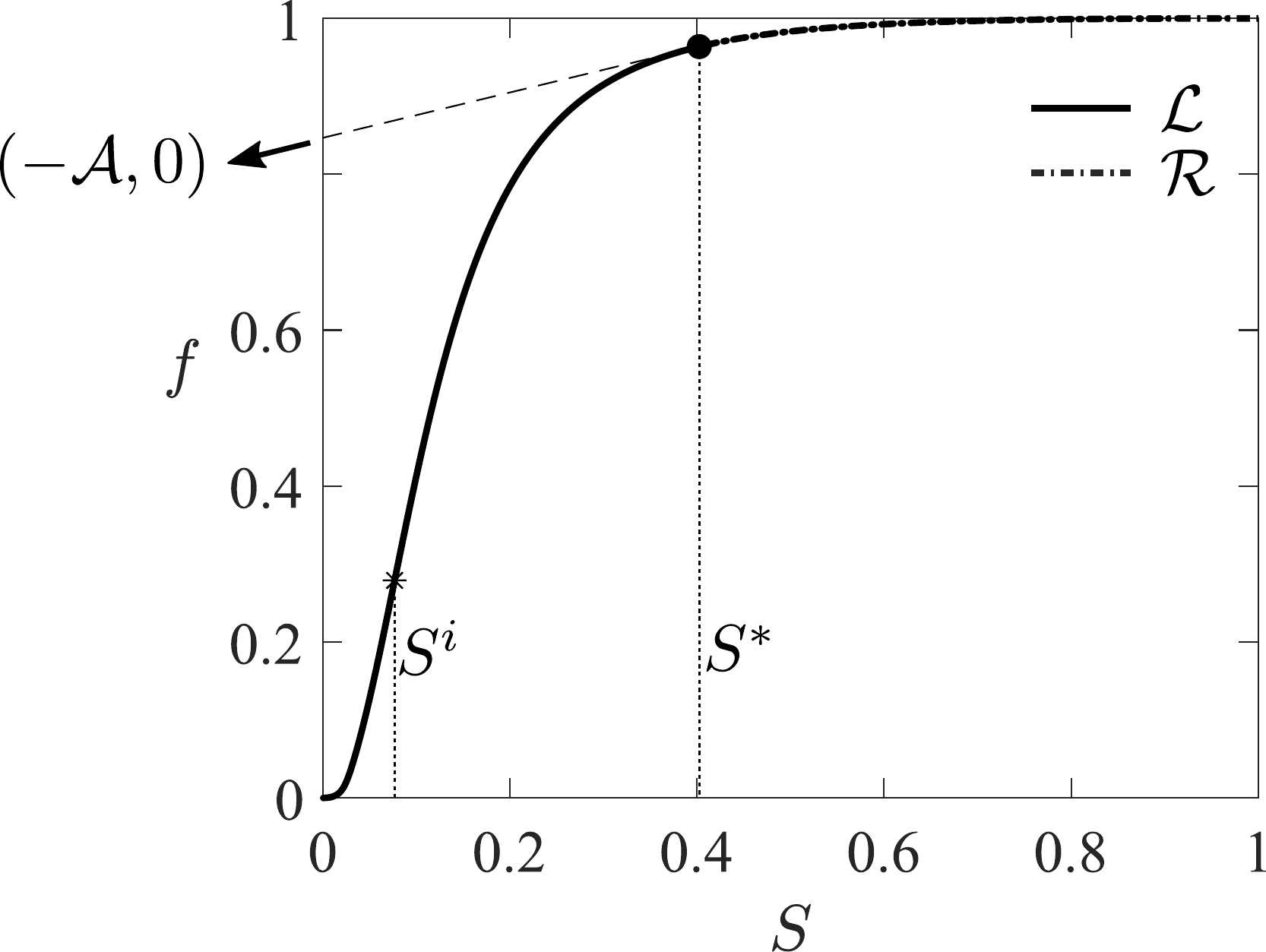}
	\caption{Fractional flow function $f$ with the inflection point and tangent line starting from point $(-\mathcal{A},0)$, which coincides with the slope of $f$ at $S^*$. The left panel presents a schematic representation. The right one corresponds to the realistic model described in Subsection~\ref{App:A} with parameter values from Table \ref{Table:Values_parameter}. 
	}
	\label{fig:diagram}
\end{figure}

We stress that the point $S^*$ defined in the proof above is widely used throughout this work.

We consider the phase plane as the set $I \times I$, described as a union of three sets: the \emph{transition curve} $\mathcal{T}$ (by Proposition \ref{pro:Transition_Lemma}), $\mathcal{L}$ on the left side of $\mathcal{T}$, and $\mathcal{R}$ on the right side of $\mathcal{T}$, see Fig.~\ref{fig:regions}. Those sets are defined as
\begin{eqnarray}
\label{eq:transition_curve}	\mathcal{T}&=&\left\{U=(S,C) \in I\times I : \lambda_S(U)=\lambda_C(U)  \right\},\\
\label{eq:set_L}	\mathcal{L}&=&\left\{U=(S,C) \in I\times I : \lambda_S(U)>\lambda_C(U) \right\},\\
\label{eq:set_R}	\mathcal{R}&=&\left\{U=(S,C) \in I\times I : \lambda_S(U)<\lambda_C(U) \right\}.
\end{eqnarray}
Notice that, the eigenvectors $r_C$ and $r_S$ (see Eqs. \eqref{eq:lambda_C}-\eqref{eq:lambda_S}) are linearly dependent in $\mathcal{T}$. Thus, the matrix $A$ is not diagonalizable in $\mathcal{T}$, similar to \cite{ELI,isaacson1986analysis, johansen1988, temple1982global}.
\begin{figure}[h!]
	\centering	    
	\includegraphics[height=5cm]{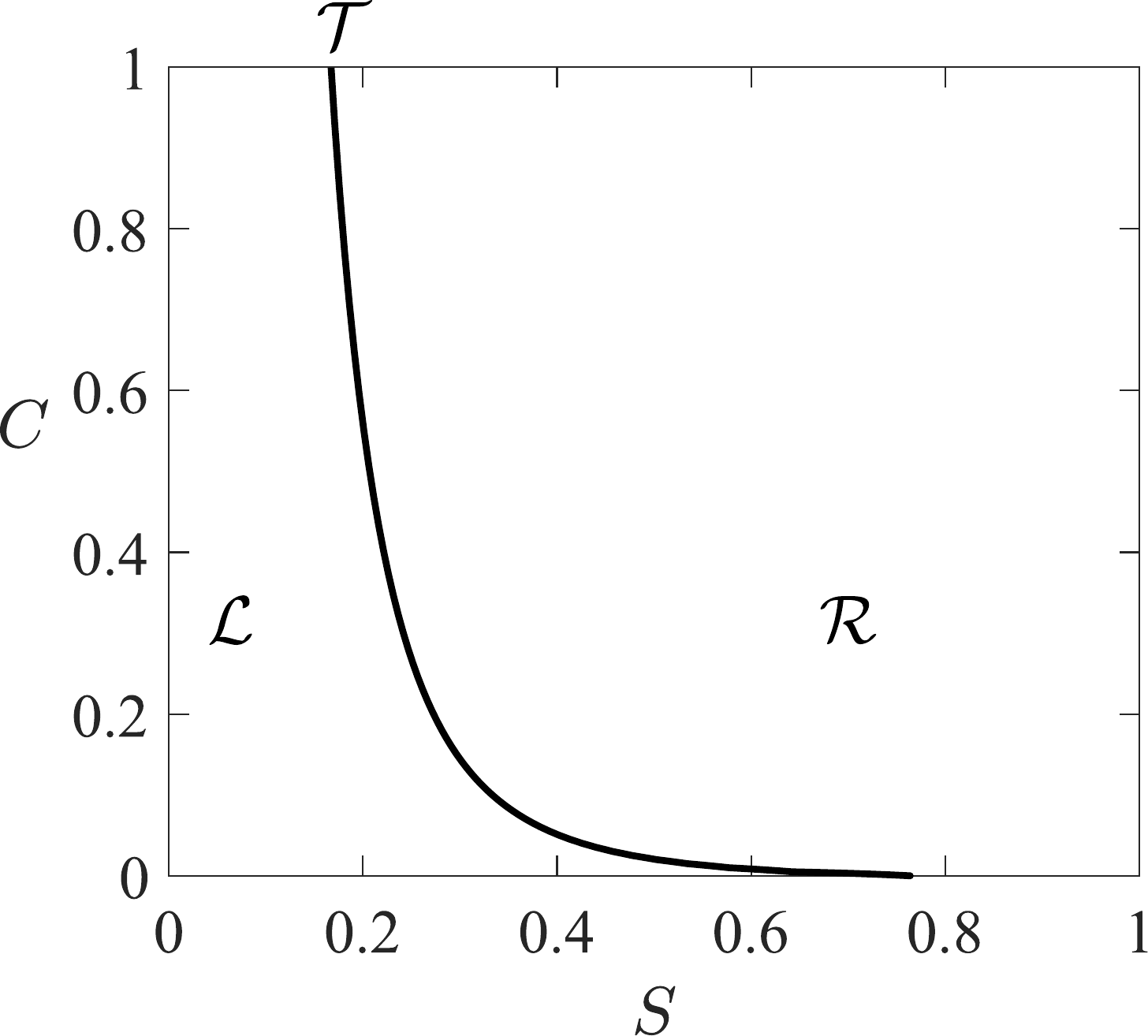}
	\caption{The phase plane is divided in sets $\mathcal{L}$ and $\mathcal{R}$ separated by transition curve $\mathcal{T}$.}
	\label{fig:regions}
\end{figure}

\subsection{Fundamental waves}
\label{sec:preliminary}
The typical solution to the Riemann problem is composed of constant states separated by waves moving with different velocities.
In this section, we review the basic wave concepts required to solve the Riemann problem \eqref{eq:mass_con_S}-\eqref{eq:evol_con_C} and initial condition \eqref{eq:initial_cond}. The states on each wave left and right are referred to as $U_L$ and $U_R$, respectively.

To determine the basic waves that compose the solution of system \eqref{eq:mass_con_S}-\eqref{eq:evol_con_C}, we classify the characteristic fields 
corresponding to variables $S$ and $C$. The $C$-characteristic field satisfies $\nabla\lambda_C \cdot r_C = 0$, classified as {\it linearly degenerate} and results in a contact discontinuity solution \cite{LeVeque1990}. The $S$-characteristic field provides $\nabla\lambda_S \cdot r_S =\partial_{SS}f$, which gives us a field with a {\it local linear degeneracies} at the inflection point of $f$ (where $\partial_{SS}f=0$) for constant $C$ \cite{Barkve1989}. The corresponding solution can be a shock discontinuity, a rarefaction wave, or a combination of both.

\subsubsection{Rarefaction waves}\label{Sub_Sec:Rarefaction}

Rarefaction waves are continuous solutions of the system \eqref{eq:conservative_form} connecting states $U_L$ and $U_R$ satisfying
\begin{equation}\label{eq:Rar_Geral}
U(x,t)=V(x/t),
\end{equation}
where $V$ is an integral curve, which is locally defined as a tangent in the direction of the eigenvector $r$ associated with the eigenvalue $\lambda$ of matrix $A$ in \eqref{eq:matrix_A}; for more details, see \cite{LeVeque1990}. The initial velocity of a rarefaction wave is $\lambda(U_L)$, and the final velocity is $\lambda(U_R)$. In the $x$-$t$ plane, rarefaction waves are characterized through a set of straight lines starting from the origin with slopes between $\lambda(U_L)$ and $\lambda(U_R)$ \cite{LeVeque1990}.
Notice that \eqref{eq:Rar_Geral} yields a constant solution along the lines with fixed $x/t$.

For the system \eqref{eq:mass_con_S}-\eqref{eq:evol_con_C}, the eigenvector of the $S$-family is $r_S=(1,0)^T$, yielding the straight line integral curve with constant $C$. Therefore, the $S$-family rarefaction waves maintain the value of $C$ constant.

\subsubsection{Shock and contact waves}
\label{Sub_Sec:Discontinuities}

Shock and contact waves are discontinuous solutions satisfying Rankine-Hugoniot (RH) condition, which provides the discontinuity's propagation velocity $\sigma$ given by 
\begin{eqnarray}
\label{RH} f(U_R) - f(U_L) &=& \sigma(S_R - S_L) , \label{RH-1} \\
\label{RH-2} f(U_R)C_R -  f(U_L)C_L &=& \sigma( (S_R + \mathcal{A})C_R - (S_L + \mathcal{A})C_L ) . 
\end{eqnarray}
In the $x$-$t$ plane, these waves are represented by a straight line starting from the origin with slope $\sigma$, which refers to the solution's discontinuity.

After some algebraic manipulations, Eqs.~\eqref{RH-1}-\eqref{RH-2} can be rewritten as 
\begin{align}
\label{Eq:vel_BL}
 &\sigma = (f(U_R)-f(U_L))/(S_R-S_L),\\
 (C_R-C_L)(S_Rf(U_L)&-S_Lf(U_R)-\mathcal{A}(f(U_R)-f(U_L)))= 0.
\end{align}
The last relation is satisfied if $C_L=C_R$, or
\begin{eqnarray}\label{Eq:Vel_Contact}
 \lambda_C(U_L)=\frac{f(U_L)}{S_L+\mathcal{A}}= \frac{f(U_R)}{S_R+\mathcal{A}}=\lambda_C(U_R).
\end{eqnarray}
When $C_L=C_R$, we obtain a shock that maintains the value of $C$ constant. 
We refer to the Buckley-Leverett type solutions containing shocks and rarefactions as $S$-waves. On the other hand, if $\lambda_C(U_L)=\lambda_C(U_R)$, a contact discontinuity occurs. We call the latter a $C$-wave.

Next, we analyze the behavior of the curve $\lambda_C(S,C(S)) =$ constant in the phase plane $S$-$C$. By differentiating it in relation to $S$, then applying the chain rule and isolating $d_SC(S)$, we obtain
\begin{equation}
d_SC(S) = (\lambda_C - \lambda_S)/(\partial_C f).
\end{equation}
Therefore, from the definition of sets $\mathcal{L}$ and $\mathcal{R}$ (see Eqs. \eqref{eq:set_L}-\eqref{eq:set_R}), it follows that
\begin{equation}\label{eq:lambda_c_para}
d_SC(S) < 0 \ 
\ \text{in }\mathcal{L}, \quad d_SC(S) > 0 \ \text{in }\mathcal{R}.
\end{equation}
Therefore, the curve $\lambda_C=$ constant as a function of $C(S)$ is decreasing in $\mathcal{L}$ and increasing in $\mathcal{R}$. If this curve reaches the minimum point inside the domain $I \times I$, we obtain $d_SC(S)=0$ on the transition curve $\mathcal{T}$ (i.e., the minimum is on $\mathcal{T})$. 
It should be noticed that this curve does not always intersect the transition curve $\mathcal{T}$. For example, for small values of $S$, the curve $\lambda_C=$ constant can occasionally be found exclusively in $\mathcal{L}$, as shown in Fig.~\ref{fig:LC_CTE}.
\begin{figure}[h!]
	\centering	    
	\includegraphics[height=5cm]{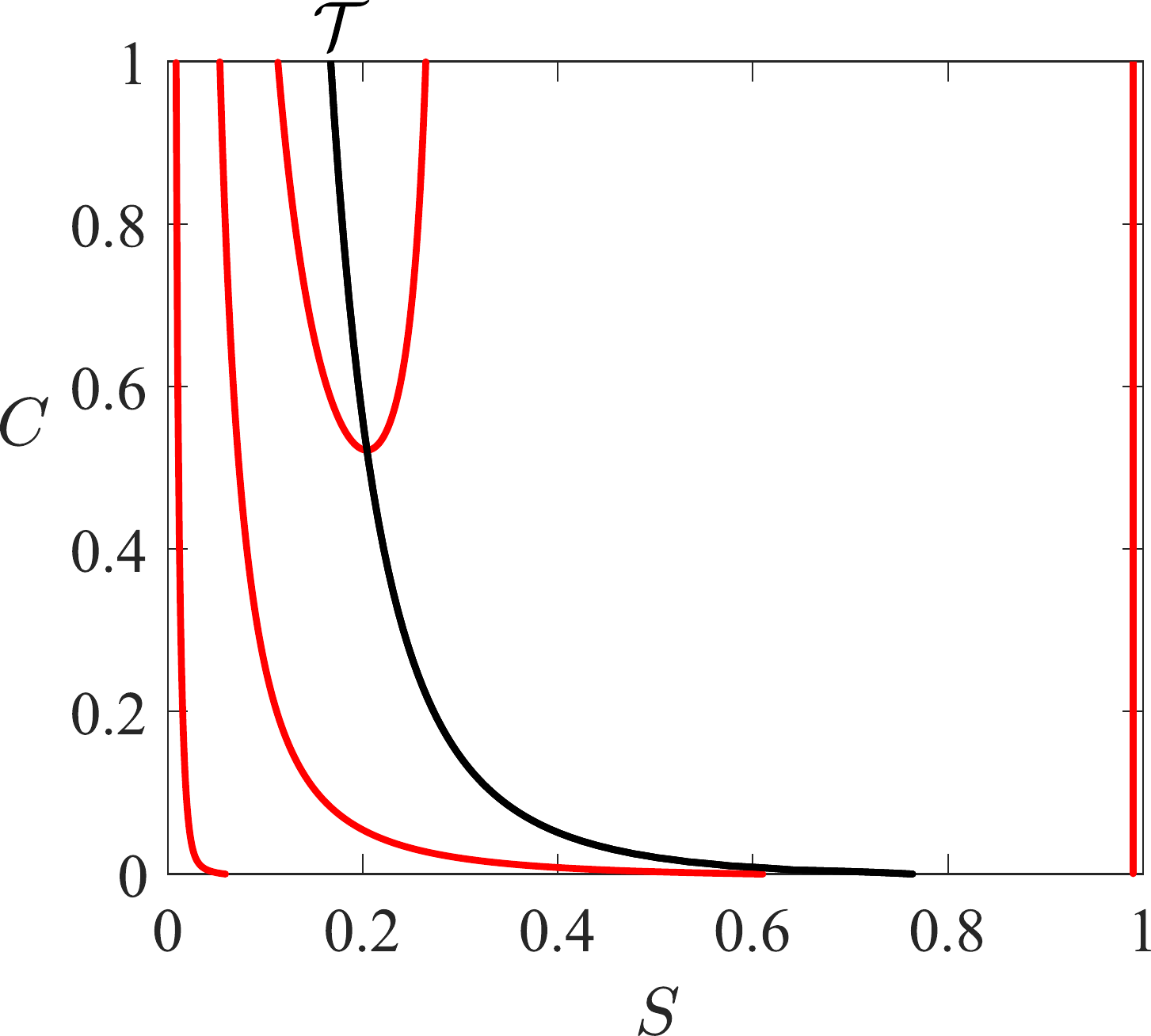}
	\caption{The phase plane with the transitional curve $\mathcal{T}$ (solid black line) and the curve $\lambda_C=$constant (solid red line) for three states $U$. This figure uses the model described in Subsection~\ref{App:A} with parameter values from Table \ref{Table:Values_parameter}.  }
	\label{fig:LC_CTE}
\end{figure}

\begin{remark}{\it
Once the shock wave solutions keep the value of $C$ constant, substituting $C_L=C_R$ in \eqref{RH}-\eqref{RH-2}, we obtain that the shock velocity is given by \eqref{Eq:vel_BL},
which is the Rankine-Hugoniot condition of the classical Buckley-Leverett equation.
On the other hand, if a contact wave occurs and we assume $C_L \ne C_R$ in \eqref{RH}-\eqref{RH-2}, its velocity is given by Eq.~\eqref{Eq:Vel_Contact},
or equivalently, $\sigma = \lambda_C(U_L)=\lambda_C(U_R)$.
}
\end{remark}

\subsection{Admissibility criteria}
\label{Sub_sec:Adm_Crit}

Generally, the solution of the system \eqref{eq:mass_con_S}-\eqref{eq:evol_con_C} with initial data \eqref{eq:initial_cond} is not unique. To obtain physically meaningful solutions, one uses entropy conditions \cite{Smoller1994}. In this section, we review the literature about entropy conditions and present the condition adopted in this work. Later, in Section \ref{subsec:numerical} we validate the obtained solutions through direct numerical simulations.

Let us consider $\sigma$ the propagation velocity of the discontinuity connecting $U_L$ and $U_R$ satisfying RH condition \eqref{RH-1}-\eqref{RH-2}. We say that a characteristic velocity $\lambda$ leaves the discontinuity if $\lambda(U_L) < \sigma$ or $\lambda(U_R)>\sigma$, and enters the discontinuity if $\lambda(U_L)>\sigma$ or $\lambda(U_R)<\sigma$.

In genuinely nonlinear fields of strictly hyperbolic systems of conservation laws, Lax's entropy condition states that a shock wave associated with the characteristic field $P$ is admissible if the $P$-characteristics enters the shock and exactly one characteristic of the other fields leaves it \cite{Lax1957_II}. For a linearly degenerate field, the contact wave is admissible if the $P$-characteristics are tangent to the discontinuity and exactly one characteristic of the other fields leaves it. For non-strictly hyperbolic systems, Keyfitz \& Kranzer \cite{keyfitz1980system} generalized Lax entropy condition, where their system possesses a linearly degenerate and a genuinely non-linear families. The entropy condition in Barkve's work \cite{Barkve1989} is equivalent to the generalized Lax entropy condition in the absence of degeneracies. Regarding polymer flooding, the works \cite{ELI, isaacson1986analysis, temple1982global} used a similar model to the one considered here. Despite one family possesses degeneracies, generalized Lax or a similar entropy condition was considered. A physically meaningful vanishing adsorption condition was presented by Petrova et al., \cite{petrova2022vanishing}, which implies the standard entropy conditions from \cite{ELI, isaacson1986analysis,keyfitz1980system}. For simplicity, in this work we use the entropy condition presented by Isaacson \& Temple \cite{isaacson1986analysis}, defined below.
\begin{definition}
    A $C$-wave is admissible if connects states on the same side of the transition curve $\mathcal{T}$. An $S$-wave is admissible if satisfies a usual entropy condition for scalar conservation laws.
\end{definition}
Once the fractional flow function $f$ possesses an inflection point, we adopt Oleinik's entropy condition for the $S$-waves. This condition states that a discontinuity is admissible if
\begin{eqnarray}\label{Eq:Oleinik}
\frac{f(U)-f(U_L)}{S-S_L} \ge \sigma \ge \frac{f(U_R)-f(U)}{S_R-S}
\end{eqnarray}
for all $U$ between $U_L$ and $U_R$, where $\sigma$ is the discontinuity velocity, given by Eq.~\eqref{Eq:vel_BL}.
Geometrically, Oleinik's entropy condition states that if $S_L < S_R$ (or $S_L>S_R$), the shock is admissible if the secant line connecting $(S_L,f(U_L))$ and $(S_R,f(U_R))$ lies above (or below) the arc of the graphic of $f$ with endpoints at $(S_L,f(U_L))$ and $(S_R,f(U_R))$ \cite{dafermos2005hyperbolic}. We emphasize that applying Olenik's entropy condition for the $S$ characteristics field is possible because all waves in this field keep the value o $C$ constant.

If $S$ tends to $S_L$ or $S_R$ in Eq.~\eqref{Eq:Oleinik}, we obtain the following relation (see \cite{dafermos2005hyperbolic} for further details):
\begin{eqnarray}\label{Eq:Lax}
\lambda_S(U_R) \le \sigma \le \lambda_S(U_L).
\end{eqnarray}
Therefore, the $S$-wave velocity is at least $\lambda_S(U_R)$ and at most $\lambda_S(U_L)$; the same result was obtained in \cite{ELI}.

\section{Construction of the Riemann solution}
\label{sec:RP_construction}

In this section, we present conditions for the existence of a solution to the Riemann problem \eqref{eq:mass_con_S}-\eqref{eq:evol_con_C}, and \eqref{eq:initial_cond}. The proof is done by constructing the solution as a sequence of waves connecting steady states. A similar construction was presented in \cite{ELI, isaacson1986analysis, johansen1988, temple1982global}.

We denote by $U_1 \xrightarrow{\quad a \quad}
U_2$, an $a$-wave connecting the initial state $U_1$ to the final state $U_2$. For the wave sequence $U_1 \xrightarrow{\quad a \quad}
U_2 \xrightarrow{\quad b \quad} U_3$,
we denote by $v_i^a$ the initial wave velocity of the $a$-wave and $v_f^b$ the final wave velocity of the $b$-wave. In this case, the wave sequence is said to be \emph{compatible} if and only if
\begin{eqnarray}\label{Cond_Compatibilidade}
v_f^a \le v_i^b.
\end{eqnarray}

\begin{theorem}\label{teo:RP}
For arbitrary states $U_L \in I\times I$ and $U_R \in I\times I$, there exists a finite sequence of compatible $S$ and $C$-waves providing a solution of the Riemann problem \eqref{eq:mass_con_S}-\eqref{eq:evol_con_C} with left state $U_L$ and right state $U_R$. The solution is not unique in the phase plane, but unique in the $x$-$t$ plane.
\end{theorem}

The proof of Theorem~\ref{teo:RP} follows from the auxiliary lemmas below. In addition to Lemmas~\ref{Lemma-1}-\ref{Lemma-D2}, we present a phase-plane classification in Lemma~\ref{Lemma:Classification} identifying sets of compatible wave sequences that fill the phase plane.
The proof of the solution's uniqueness in the $x$-$t$ plane and examples of non-uniqueness are presented in Section~\ref{Sec:Uniqueness}.

From now on, we refer to the injection and initial conditions of the Riemann problem as $U_L$ and $U_R$, respectively.
For a given state $U$, we define $S^k = S^k(U)$ as a set of possible values of $S$ satisfying
\begin{equation}\label{S_k}
\lambda_C(U) = \lambda_C(S^k,C).
\end{equation}
Geometrically, relation \eqref{S_k} implies that both states are in the same fractional flow function, and the secant line that connects them to $(-\mathcal{A},0)$ is the same, see Fig.~\ref{fig:S_k}. If there exists $S^k$ satisfying \eqref{S_k}, the states $U$ and $(S^k(U),C)$ are in opposite sides of the transitional curve $\mathcal{T}$ (see Figs.~\ref{fig:S_k} and\ref{fig:diagram}). In case such $S^k(U)$ does not exist, we assume $S^k(U) = + \infty$. If $S = S^*(C)$, we take $S^k(U) = S^*$ ($S^*$ was defined in Prop.~\ref{pro:Transition_Lemma}). 
\begin{figure}[h!]
	\centering	    
	\includegraphics[height=5cm]{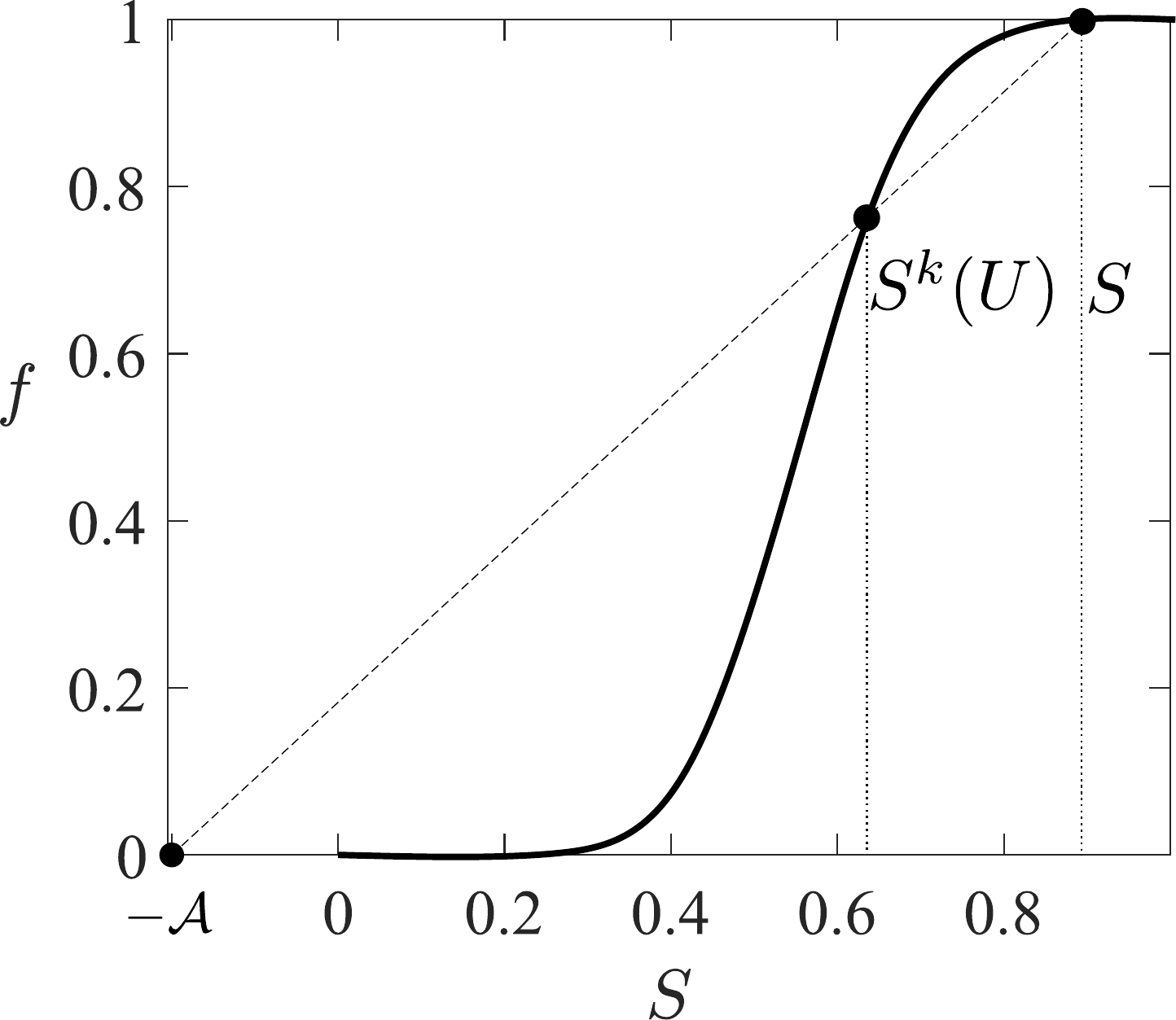}\hfill
	\includegraphics[height=5cm]{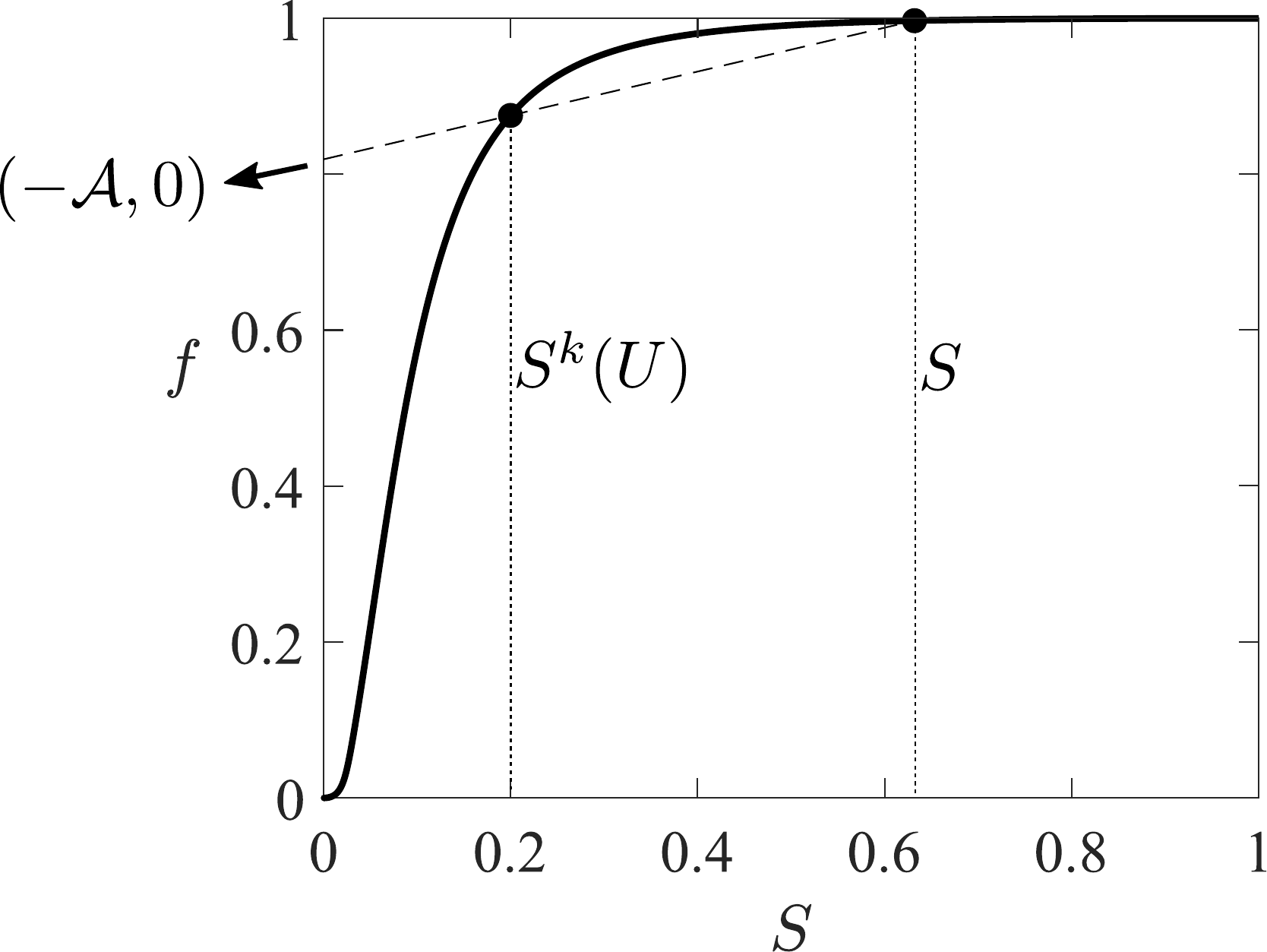}
	\caption{Geometric representation of $S^k$ for a fixed state $U$. It is the intersection between the dashed-line connecting $(-\mathcal{A}, 0)$ to $(S,f(U))$ and the fractional flow function $f(\cdot,C)$. On the left, we present a schematic representation for clarity. On the right, we use the model described in Subsection~\ref{App:A} with parameter values from Table \ref{Table:Values_parameter}.}
	\label{fig:S_k}	
\end{figure}

Hereafter, all figures use the fractional flow function for the model described in Subsection~\ref{App:A} with parameter values from Table \ref{Table:Values_parameter}.

\begin{lemma}\label{Lemma-1}
Let us consider the wave sequence
\begin{eqnarray}
U_L \xrightarrow{\quad C \quad}
U_M \xrightarrow{\quad S \quad}
U_R,
\end{eqnarray}
with $U_L \ne U_R$. This wave sequence is compatible if and only if $U_M \in \mathcal{L}\cup \mathcal{T}$ and $0\le S_R \le S^k(U_M)$.
\end{lemma}

\begin{proof}
The wave sequence is compatible if $\lambda_C(U_M) = v_f^c\le v_i^S \le \lambda_S(U_M)$, implying that it is necessary to possess $U_M \in \mathcal{L}\cup \mathcal{T}$. Analyzing the possible initial velocities of the $S$-wave, we conclude that the additional condition to guarantee the wave compatibility is that $0\le S_R \le S^k(U_M)$. Therefore, the solution is compatible if, and only if $U_M \in \mathcal{L}\cup \mathcal{T}$ and $0\le S_R \le S^k(U_M)$. 
\end{proof}



\begin{lemma}\label{Lemma-2}
		Let us consider the wave sequence
		\begin{eqnarray}
		U_L \xrightarrow{\quad S \quad}
		U_M \xrightarrow{\quad C \quad}
		U_R,
		\end{eqnarray}
		with $U_L \ne U_R$. The wave sequence is compatible if and only if $U_M \in \mathcal{R}\cup \mathcal{T}$ and $S^k(U_M) \le S_L \le 1$.
\end{lemma}
	
\begin{proof} 
Analogous to the previous lemma, a necessary condition to guarantee the wave compatibility is that $U_M \in \mathcal{R}\cup\mathcal{T}$.
Performing an analysis of the $S$-wave final velocity 
, we obtain that the wave sequence is compatible if and only if $U_M \in \mathcal{R}\cup\mathcal{T}$ and $S^k(U_M)\le S_L\le 1$. 
\end{proof}

The following lemmas analyze the possible sequences of $S$-wave, followed by a $C$-wave and an $S$-wave again for $C_L<C_R$ and $C_R<C_L$, respectively. We omit the proofs as they are an application of lemmas \ref{Lemma-1} and \ref{Lemma-2}.

\begin{lemma}\label{Lemma-D1}
    Let us consider $C_L < C_R$ and the wave sequence
    \begin{eqnarray}
    U_L \xrightarrow{\quad S_1 \quad}
    U_1 \xrightarrow{\quad C \quad}
    U_2 \xrightarrow{\quad S_2 \quad}
    U_R,
    \end{eqnarray}
    with $U_L \ne U_R$. This wave sequence is compatible if and only if all the following conditions are satisfied: $U_L\in\mathcal{R}$, $U_1\in\mathcal{T}$, $U_2\in\mathcal{L}$ and $0\le S_R\le S^k(U_2)$.
\end{lemma}

\begin{lemma}\label{Lemma-D2}
    Let us consider $C_R < C_L$ and the wave sequence
    \begin{eqnarray}
    U_L \xrightarrow{\quad S_1 \quad}
    U_1 \xrightarrow{\quad C \quad}
    U_2 \xrightarrow{\quad S_2 \quad}
    U_R,
    \end{eqnarray}
    with $U_L \ne U_R$. The wave sequence is compatible if and only if all the following conditions are satisfied: $S^k(U_1) \le S_L \le 1$, $U_1\in\mathcal{R}$, $U_2\in\mathcal{T}$ and $U_R \in \mathcal{L}$.
\end{lemma}



Lemmas~\ref{Lemma-1}, \ref{Lemma-2}, \ref{Lemma-D1}, and \ref{Lemma-D2} allow us to construct compatible solution sequences for any Riemann problem \eqref{eq:mass_con_S}-\eqref{eq:initial_cond}. In the literature, this type of result is also known as a classification lemma, as it allows us to divide the phase plane into sets according to the solution type.

\begin{lemma}\label{Lemma:Classification}
Let $U_L$ be a constant left state in the $S$-$C$ phase plane. Then, for every right state $U_R$, there is a solution of the Riemann problem \eqref{eq:mass_con_S}-\eqref{eq:initial_cond}, given by a compatible sequence of $S$ and $C$-waves.
\end{lemma}

\begin{proof}
Given a fixed state $V \in I\times I$, we define the set of states with the same $\lambda_C$ as $V$:
\begin{eqnarray}
\Gamma(V) = \{ U\in I \times I:\lambda_C(U)=\lambda_C(V) \}.
\end{eqnarray}
It should be noticed that, due to the smoothness of $\lambda_C$, $\Gamma(V)$ is a curve.
Let us fix $U_L \in I\times I$, and consider two cases: $U_L\in \mathcal{L}\cup \mathcal{T}$ and $U_L \in \mathcal{R}$.

\begin{figure}[ht!]
     \centering
     \begin{subfigure}[h!]{0.49\textwidth}
         \centering
         \includegraphics[height=5cm]{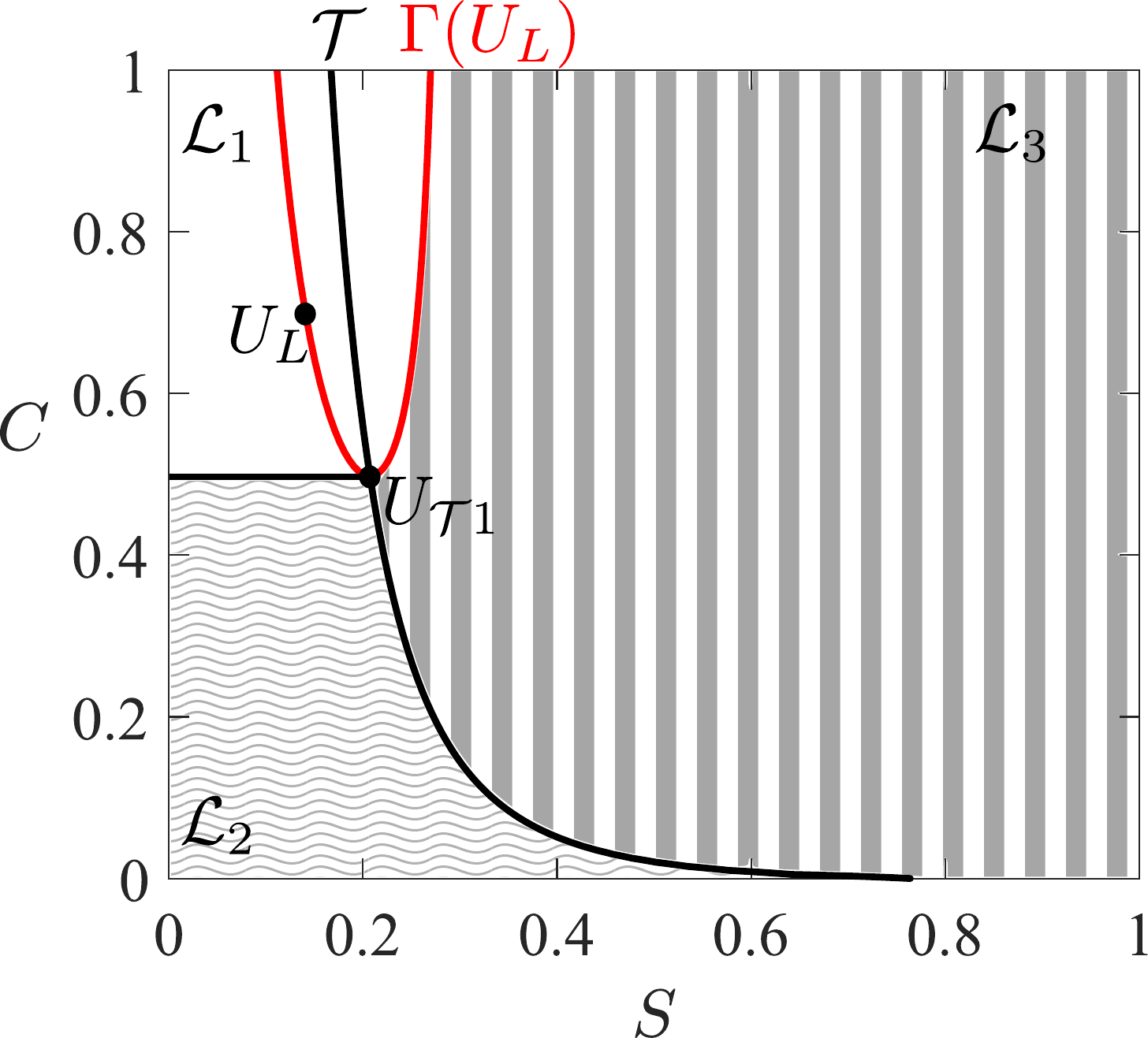}
        \caption{Sets $\mathcal{L}_i$ with $i=1,2,3$, in the phase plane following Eqs.~\eqref{eq:L1} through \eqref{eq:L3}.  }
	\label{fig:L}
     \end{subfigure}
     \begin{subfigure}[h!]{0.49\textwidth}
         \centering
         \includegraphics[height=5cm]{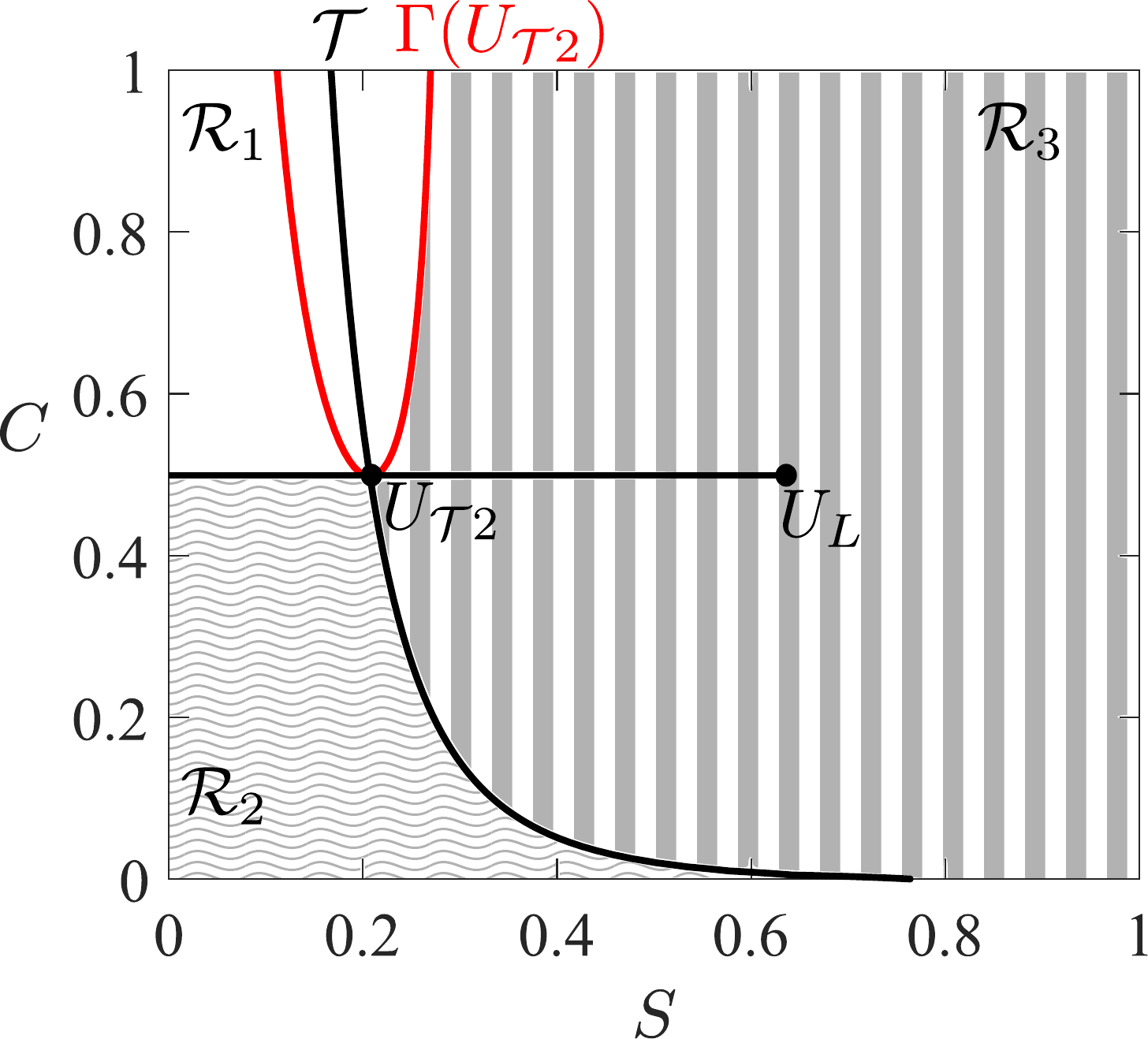}
        \caption{Sets $\mathcal{R}_i$ with $i=1,2,3$, in the phase plane following Eqs.~\eqref{eq:R1} through \eqref{eq:R3}.  }
	\label{fig:R}
     \end{subfigure}
        \caption{Phase plane division for $U_L \in \mathcal{L}\cup\mathcal{T}$ and $U_R \in \mathcal{R}$.}
        \label{fig:three graphs}
\end{figure}


\begin{enumerate}
\item \underline{In case of $U_L \in \mathcal{L} \cup \mathcal{T}$}.
We define the state $U_{\mathcal{T}1}$ as the intersection of the curve $\Gamma(U_L)$ and the transition curve $\mathcal{T}$, given by
\begin{eqnarray}
U_{\mathcal{T}1}=\mathcal{T}\cap \Gamma(U_L).
\end{eqnarray}
Observe that this state only exists if the minimum of the curve $\Gamma(U_L)$ is in the phase plane. For the cases where this intersection does not happen, we define $U_{\mathcal{T}1} = (1,0)$. In the particular case of $U_L \in \mathcal{T}$, we obtain $U_{\mathcal{T}1}=U_L$.

We divide the phase plane into three regions (see Fig.~\ref{fig:L}), defined by
\begin{eqnarray}
\label{eq:L1} \mathcal{L}_1 &=& \{ U \in I \times I : \lambda_C(U)\geq\lambda_C(U_L) \} \cup  \{U \in \mathcal{L} : C\ge C_{\mathcal{T}1}\} ,\\
\label{eq:L2} \mathcal{L}_2 &=& \mathcal{L} - \mathcal{L}_1 \cup \{U \in  \mathcal{L} :  C=C_{\mathcal{T}1} \},\\
\label{eq:L3} \mathcal{L}_3 &=& I \times I - \mathcal{L}_1 - \mathcal{L}_2 \cup \{ U \in \mathcal{R} : \lambda_C(U) = \lambda_C(U_L) \}.
\end{eqnarray}

  If $U_R \in \mathcal{L}_1$, the solution construction follows Lemma~\ref{Lemma-1}, where the intermediary state $U_M$ is defined by 
\begin{eqnarray}\label{eq:UM_L1}
U_M=\Gamma(U_L)\cap \{ U \in \mathcal{L}: C=C_R\}.
\end{eqnarray}    
   
If $U_R \in \mathcal{L}_2$, the solution construction follows Lemma~\ref{Lemma-D2}. In this case, the solution possesses two intermediate states, given by
\begin{eqnarray}
\label{eq:U2_L2}   U_2 &=&\mathcal{T} \cap \{ U \in I \times I: C=C_R \},\\
\label{eq:U1_L2}   U_1 &=& \Gamma(U_2)\cap \{ U \in \mathcal{R}: C=C_L\}.
\end{eqnarray}

If $U_R \in \mathcal{L}_3$, the solution construction is given by Lemma~\ref{Lemma-2}. The intermediary state $U_M$ is defined as
\begin{eqnarray}\label{eq:UM_L3}
U_M=\Gamma(U_R) \cap \{ U \in \mathcal{R}: C=C_L\}.
\end{eqnarray}

\item \underline{In case of $U_L\in \mathcal{R}$}. 
We define the state $U_{\mathcal{T}2}$ as the intersection of the transition curve $\mathcal{T}$ and the line $C=C_R$, as shown in Fig.~\ref{fig:R} and described by 
\begin{eqnarray}
U_{\mathcal{T}2} = \mathcal{T} \cap \{U \in I \times I: C = C_L\}.
\end{eqnarray}
It is possible to identify three sets (see Fig.~\ref{fig:R}) in the phase plane given by 
\begin{eqnarray}
\label{eq:R1}    \mathcal{R}_1 &=& \{ U\in I \times I: \lambda_C(U)\ge\lambda_C(U_{\mathcal{T}2}) \} \cup  \{ U \in \mathcal{L}: C >  C_L \},\\
\label{eq:R2}    \mathcal{R}_2 &=& \mathcal{L}-\mathcal{R}_1 ,\\
\label{eq:R3}    \mathcal{R}_3 &=& I \times I -\mathcal{R}_1 - \mathcal{R}_2 \cup \{U \in \mathcal{R}: \lambda_C(U)=\lambda_C(U_{\mathcal{T}2})\}.
\end{eqnarray}

          If $U_R \in \mathcal{R}_1$, the solution construction is given by Lemma~\ref{Lemma-D1}. Two intermediary states are defined by
\begin{eqnarray}
\label{eq:U1_R1}    U_1 &=& U_{\mathcal{T}2},\\
\label{eq:U2_R1}    U_2 &=& \Gamma(U_1) \cap \{ U \in \mathcal{L} : C = C_R\}.
\end{eqnarray}

If $U_R \in \mathcal{R}_2$, the solution construction follows Lemma~\ref{Lemma-D2}. The two intermediate states of this solution are listed below
\begin{eqnarray}
\label{eq:U2_R2}    U_2 &=& \mathcal{T} \cap \{U\in I \times I: C = C_R\},\\
\label{eq:U1_R2}    U_1 &=& \Gamma(U_2) \cap \{ U \in \mathcal{R} : C = C_L\}.
\end{eqnarray}

If $U_R \in \mathcal{R}_3$, the solution is composed of a wave sequence given by Lemma~\ref{Lemma-2}. The intermediate state $U_M$ is defined as follows
\begin{equation}
\label{eq:UM_R3}
U_M=\Gamma(U_R) \cap \{ U \in \mathcal{R}: C=C_L\}.
\end{equation}

\end{enumerate}

 \end{proof}

\begin{remark}
\label{rem:set_division}
{\it Proofs of Lemmas~\ref{Lemma-1}-\ref{Lemma:Classification} strongly depend on the states $U_L$ and $U_R$. 
Notice that, for limiting cases (small values of $S$, small or big values of $C$), it is possible to have some of these sets empty. Despite the occasional changes in the proofs, the main results remain valid, the solution's existence for any values $U_L$ and $U_R$ in the phase plane $S$-$C$.}
\end{remark}

\section{On the well-posedness of the problem}\label{Sec:Uniqueness}

For the Riemann problem solutions in this work the existence is proved by construction and the uniqueness is considered in the following sense: given a left state $U_L$ and a right state $U_R$ in the phase plane $I\times I$, there exists a unique compatible wave sequence connecting them.



\subsection{The intersection between sets $\mathcal{L}_1$ and $\mathcal{L}_2$}
\label{sec:L1L2}

In this case, $U_L\in\mathcal{L}$ and $U_R$ is in the intersection between sets $\mathcal{L}_1$ and $\mathcal{L}_2$ (see Fig.~\ref{fig:L}).
Possible solutions are given in Lemmas~\ref{Lemma-1} and \ref{Lemma-D2}.
The solution corresponding to Lemma~\ref{Lemma-D2} presents two intermediate states, $U_1$ and $U_2$. State $U_2$ coincides with $U_{\mathcal{T}1}$, and $U_1$ possesses the same $\lambda_C$ as $U_L$ with $C=C_L$, yielding $\lambda_C(U_L) = \lambda_C(U_1) = \lambda_C(U_2)$. Geometrically, these relations imply that the points $(S_L,f(U_L))$, $(S_1,f(U_1))$ and $(S_2,f(U_2))$ are on the same secant line connecting $(-\mathcal{A},0)$ to $(S_L,f(U_L))$, see Fig.~\ref{fig:UniGeo_31}.
\begin{figure}[h!]
	\centering	    
	\includegraphics[height=5cm]{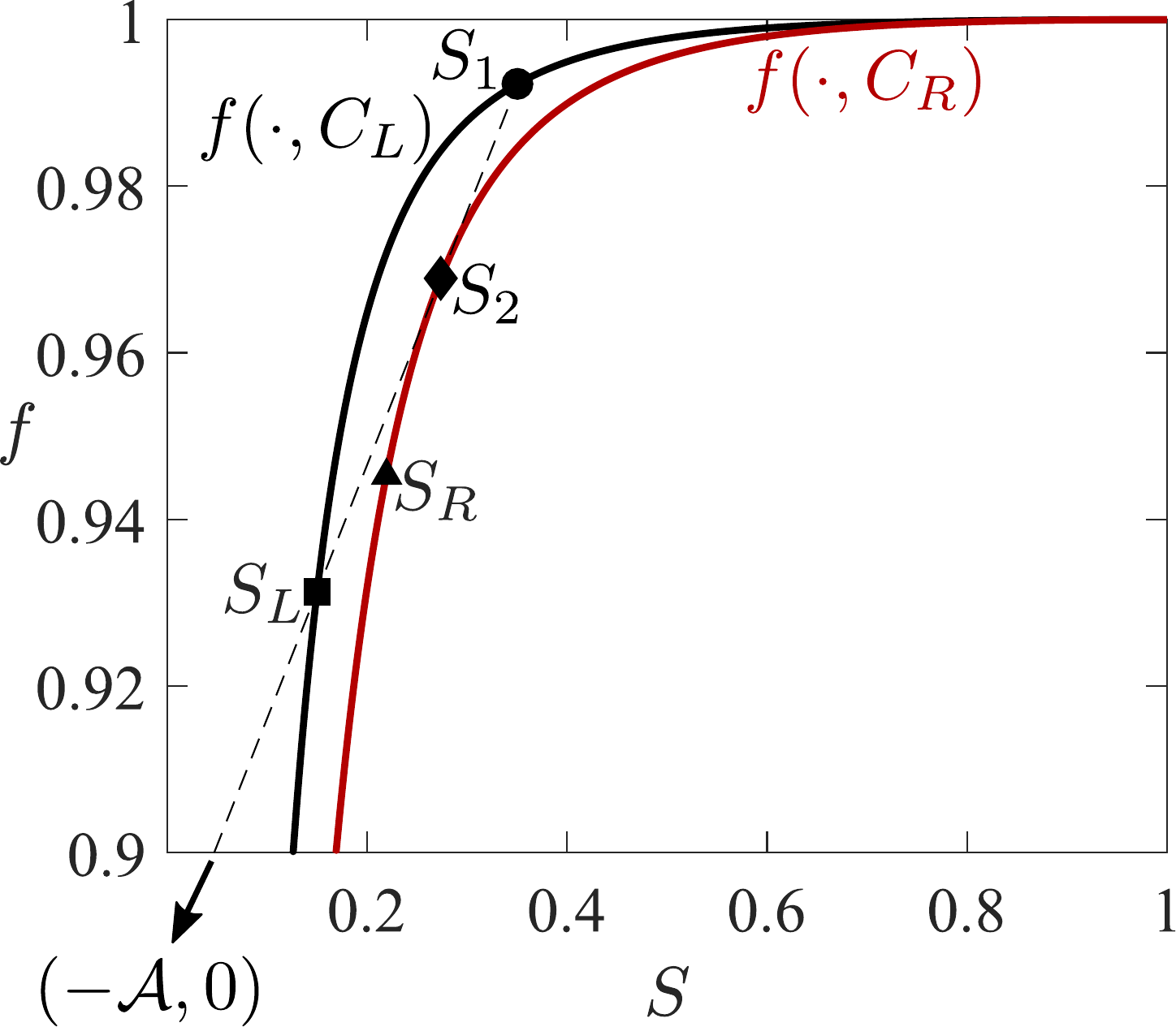}
	\caption{Fractional flow function of the states $U_L$ and $U_R$ with the secant line that connects $(-\mathcal{A},0)$ with $(S_L,f(U_L))$. The points $(S_1,f(U_1))$ and $(S_2,f(U_2))$ are in the same secant line (black dashed line) that connects $(-\mathcal{A},0)$ and $(S_L,f(U_L))$.}
	\label{fig:UniGeo_31}
\end{figure}

This intersection corresponds to $U_L \in \mathcal{L}$ and, from Lemma \ref{Lemma-D2}, $U_1 \in \mathcal{R}$ and $U_2 \in \mathcal{T}$, yielding
\begin{equation}\label{eq:ineq_1}
S_L \le S_2 \le S_1,
\end{equation}
where the equality only happens when $U_L \in \mathcal{T}$. 
The states $U_L$ and $U_1$ satisfying Eq.~\eqref{eq:ineq_1} are connected by the secant line located beneath the fractional flow function of $C_L$; thus, they are connected by a shock $S$-wave.
Therefore, the contact and the shock waves possess the same velocity; see Fig.~\ref{fig:UniGeo_31}.
\begin{figure}[h!]
	\centering	    
	\includegraphics[height=5cm]{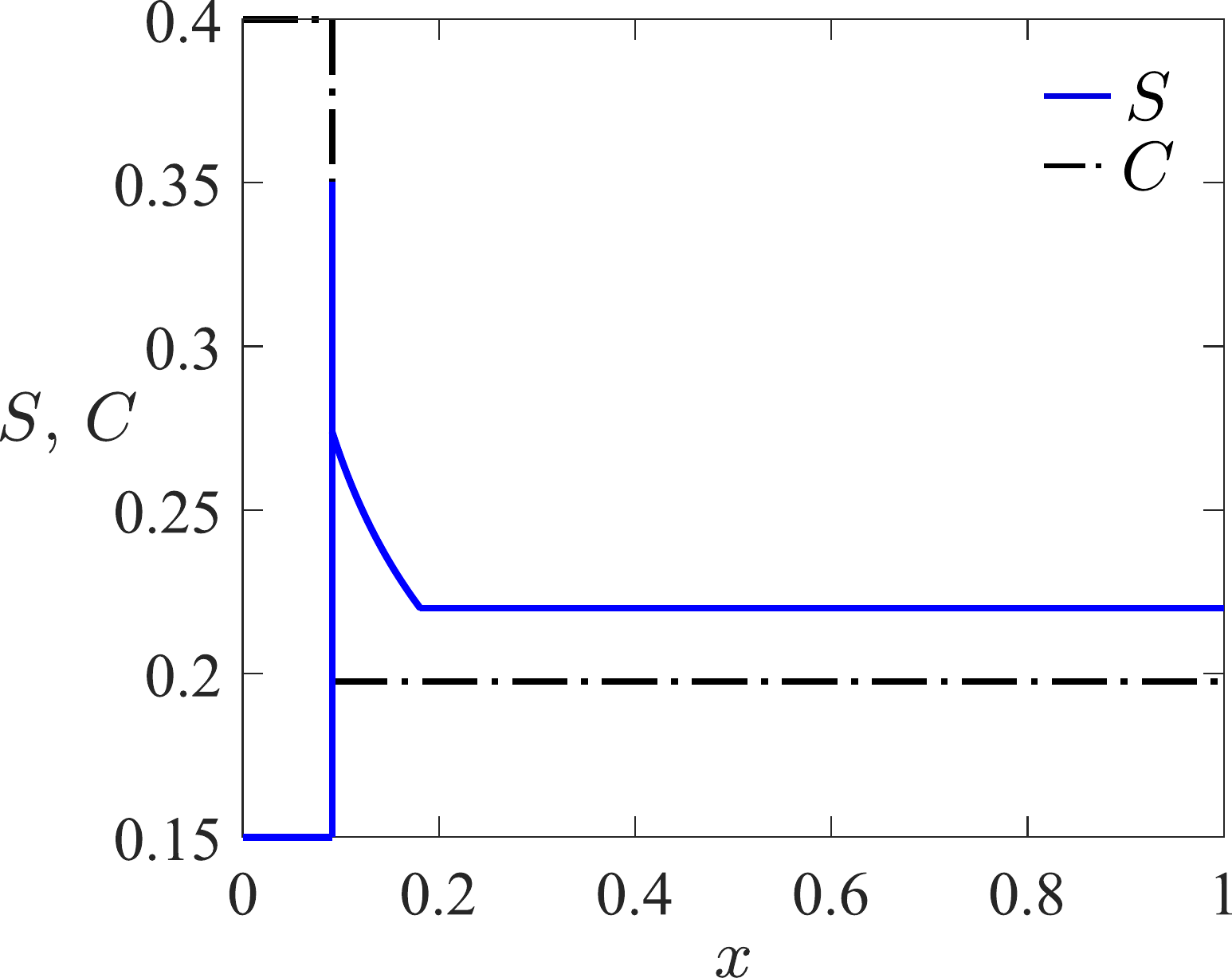}\hfill
	\includegraphics[height=5cm]{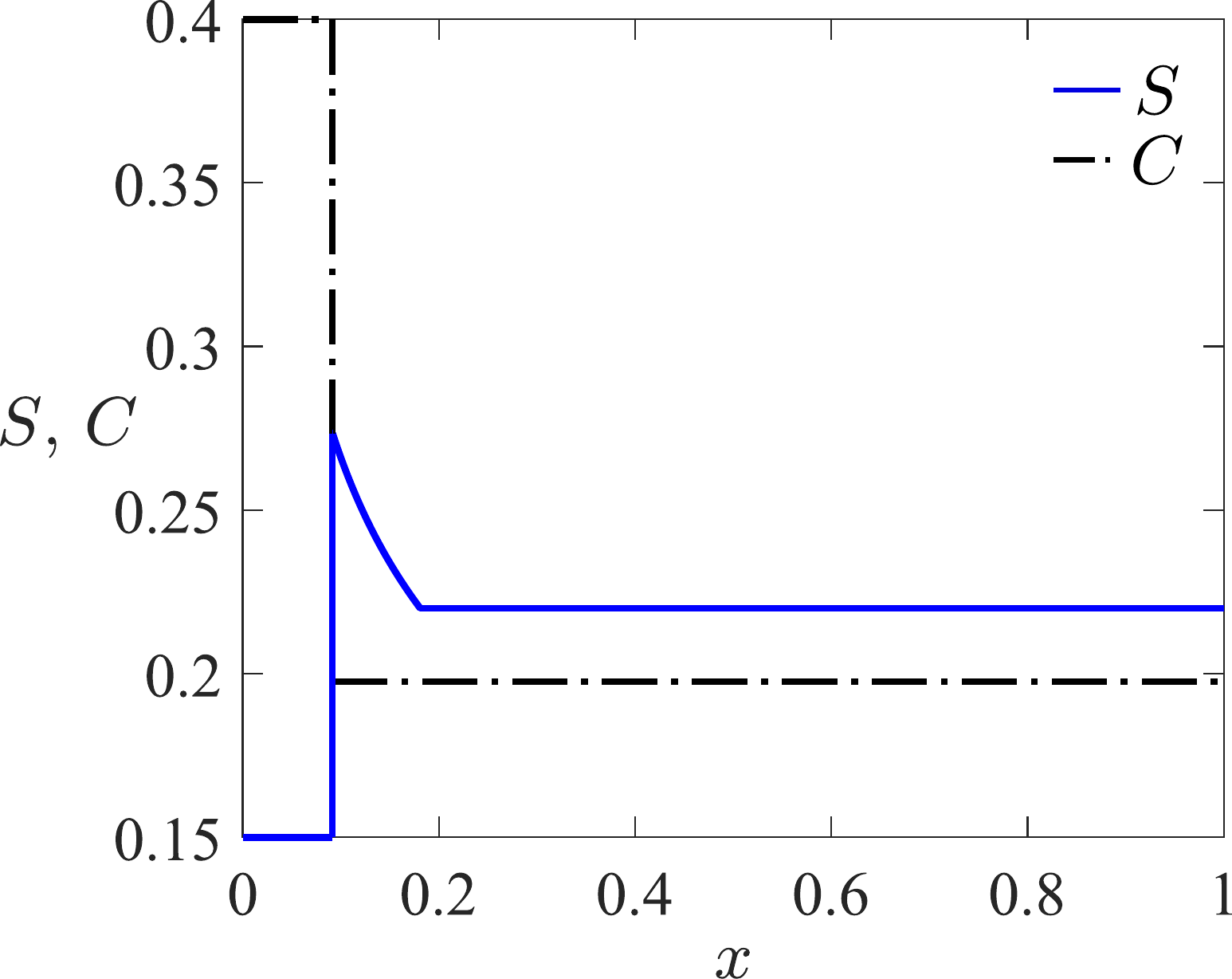}
	\caption{Solution profile at the same point in $\mathcal{L}_1 \cap \mathcal{L}_2$. The left panel shows the solution described in Lemma~\ref{Lemma-D2} and the right panel shows the one described in Lemma~\ref{Lemma-1}. This plot was done for the model explained in Section \ref{sec:application}, $U_L=(0.15,0.4)$, $U_R=(0.22,0.1975540160002906)$, time $t=0.3$ and the parameter values from Table~\ref{Table:Values_parameter}.}
	\label{fig:Unicidade_31}
\end{figure}

On the other hand, the solution following Lemma~\ref{Lemma-1} possesses an intermediate state $U_M=U_{\mathcal{T}1}=U_2$, which is connected to $U_L$ through a $C$-wave. Thus, this $C$-wave possesses the same velocity as the sequence of $S$ and $C$ waves of the previous solution. Despite this, as $S_M = S_2 < S_1$ (see Eq.~\eqref{eq:ineq_1}), the saturation profile for each construction is not the same. Fig.~\ref{fig:Unicidade_31} shows an example of these two possible saturation profiles for this case.


\subsection{The intersection between sets $\mathcal{L}_1$ and $\mathcal{L}_3$}
\label{sec:L1L3}

In this case, $U_L\in\mathcal{L}$ and $U_R$ is in the intersection between sets $\mathcal{L}_1$ and $\mathcal{L}_3$ (see Fig.~\ref{fig:L}).
Figure~\ref{fig:Unicidade_PF} shows the possible wave sequences in the phase plane, following Lemma~\eqref{Lemma-1} with intermediate state $U_M^1\in \mathcal{L}$, and following Lemma~\eqref{Lemma-2} with intermediate state $U_M^2 \in \mathcal{R}$; yielding $S_L < S_M^2$ and $S_M^1 < S_R$.
Geometrically, this scenario occurs when the states $U_L$, $U_R$ and $U_M^1$, $U_M^2$ are in the same secant line connecting $(-\mathcal{A},0)$ and $(S_L,f(U_L))$ and posses the same value of $\lambda_C$.
Therefore, the $S$-waves connecting $U_L$ to $U_M^2$ and $U_M^1$ to $U_R$ are always shock waves with the same propagation velocity as the contact wave.
\begin{figure}[h!]
	\centering	    
	\includegraphics[height=5cm]{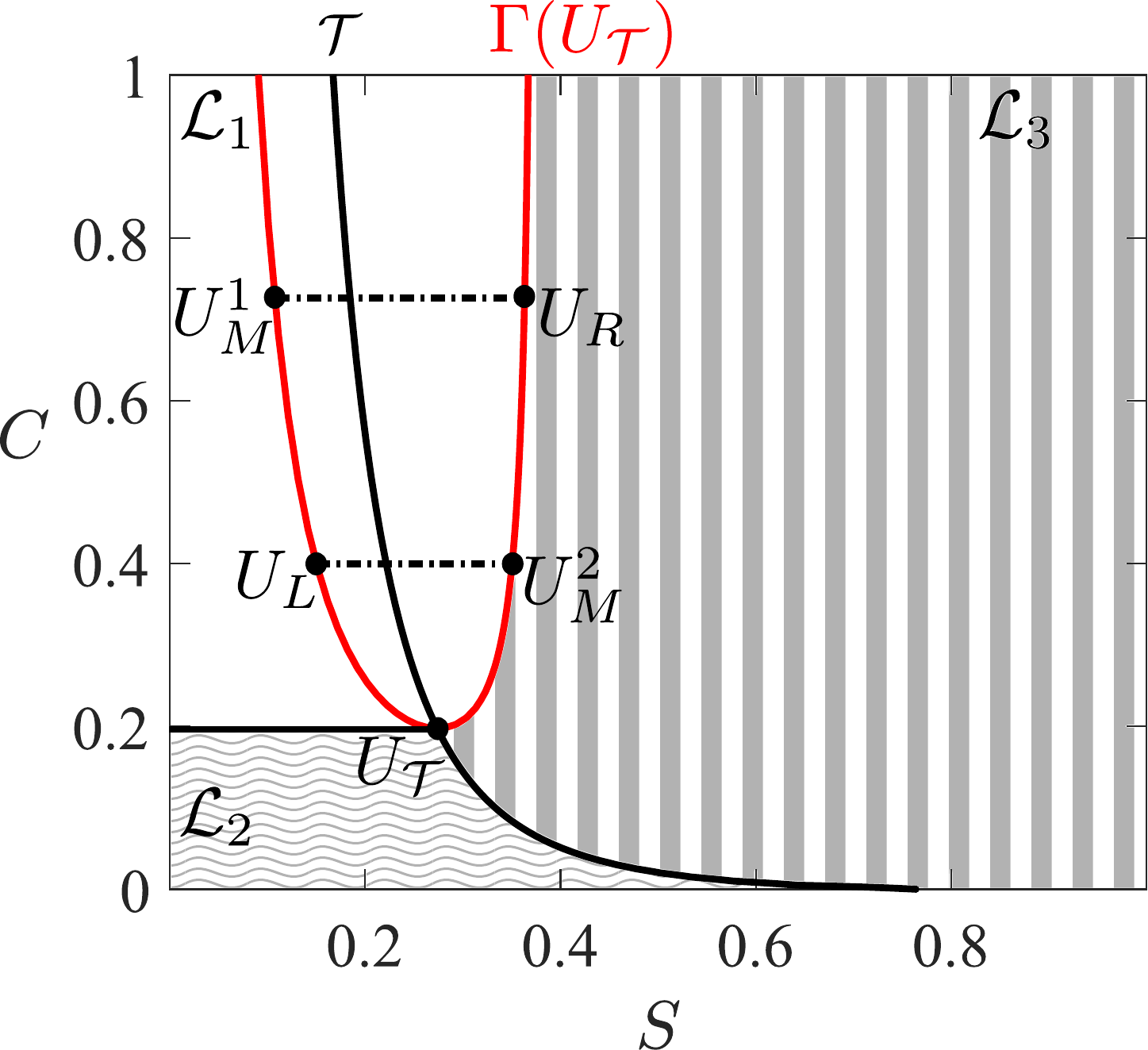}
	\caption{Two compatible wave sequences are possible with state $U_L \in \mathcal{L}_1$ and state $U_R \in \mathcal{L}_1 \cap \mathcal{L}_3$: one links $U_L$ to $U_M^1$ by a $C$-wave, then $U_M^1$ to $U_R$ by a $S$-wave. The other one links  $U_L$ to $U_M^2$ by a $S$-wave, then $U_M^2$ to $U_R$ by a $C$-wave.}
	\label{fig:Unicidade_PF}
\end{figure}

Despite the equality in the $S$ and $C$-waves velocities, the saturations $S_L,\ S_M^1, \ S_M^2$ and $S_R$ are always different, resulting in different profiles of the solution. Figure~\ref{fig:Unicidade_11} shows an example of the possible profiles solution for this case.
\begin{figure}[h!]
	\centering	    
	\includegraphics[height=5cm]{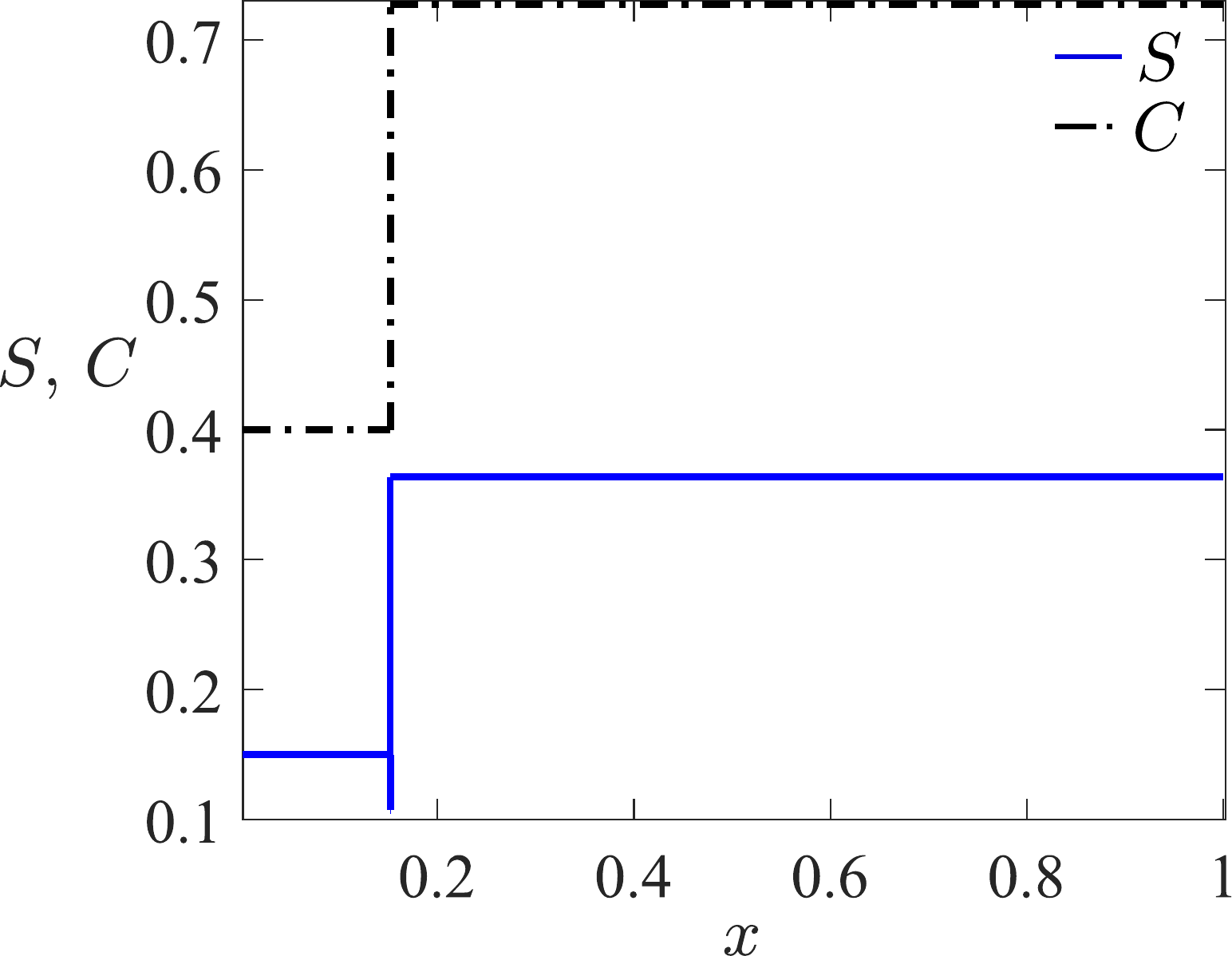}\hfill
	\includegraphics[height=5cm]{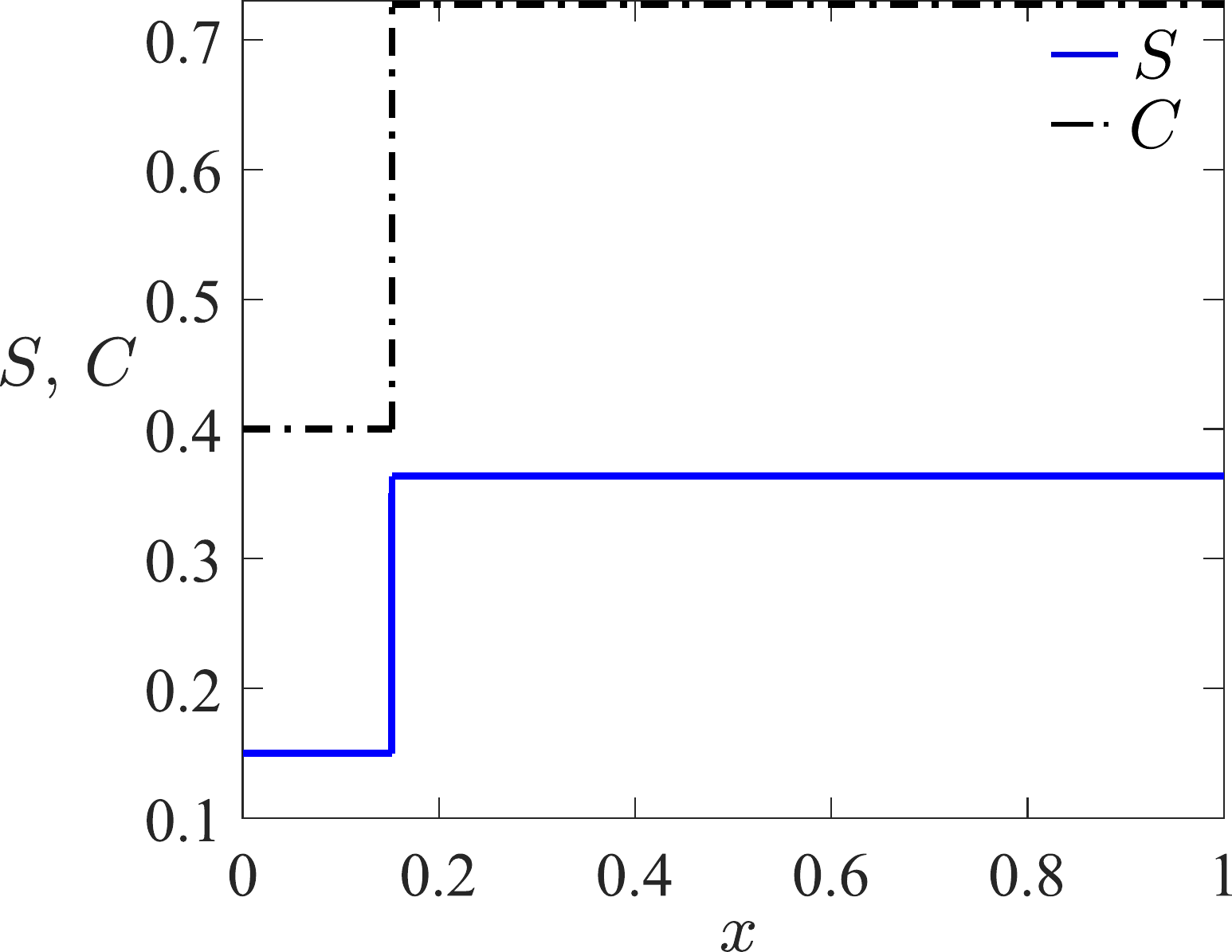}
	\caption{Solution profile for $U_R \in \mathcal{L}_1 \cap \mathcal{L}_3$, with $C_L<C_R$. The left panel shows the solution described in Lemma~\ref{Lemma-1} and the right panel the one described in Lemma~\ref{Lemma-2}. This plot was made using the model from Section~\ref{sec:application}, with $U_L=(0.15,0.4)$, $U_R=(0.3636,0.7273)$, time $t=0.3$ and parameters values from Table~\ref{Table:Values_parameter}.}
	\label{fig:Unicidade_11}
\end{figure}

\subsection{The intersection between sets $\mathcal{R}_1$ and $\mathcal{R}_3$}
\label{sec:R1R3}

In this case, $U_L\in\mathcal{R}$ and $U_R$ is in the intersection between sets $\mathcal{R}_1$ and $\mathcal{R}_3$ (see Fig.~\ref{fig:R}).
Possible solutions follow Lemmas~\ref{Lemma-2}, \ref{Lemma-D1}. The solution provided by Lemma~\ref{Lemma-D1} possesses two intermediate states: $U_1=U_{\mathcal{T}2}$ and $U_2 \in \mathcal{L}$ yielding $S_2 < S_1 < S_R$. Moreover, $\lambda_C(U_2) = \lambda_C(U_1)=\lambda_C(U_R)$, which geometrically indicates that points $(S_1,f(U_1))$, $(S_2,f(U_2))$ and $(S_R,f(U_R))$ are in the same secant line connecting $(-\mathcal{A},0)$ to $(S_R,f(U_R))$. Thus, the $S$-wave connecting $U_2$ to $U_R$ is a shock with the same propagation velocity as the $C$-wave connecting $U_1$ to $U_2$.

The solution construction following Lemma~\ref{Lemma-2} has an intermediate state $U_M=U_{\mathcal{T}2}$. Once the $C$-wave speed is the same as the sequence of $S$ and $C$-waves presented in Lemma~\ref{Lemma-D1}, both wave sequences move together with time. However, as $S_2 < S_1=S_M$, the solution profile is different for each wave sequence, as shown in Fig.~\ref{fig:Unicidade_R14}.
\begin{figure}[h!]
	\centering	    
	\includegraphics[height=5cm]{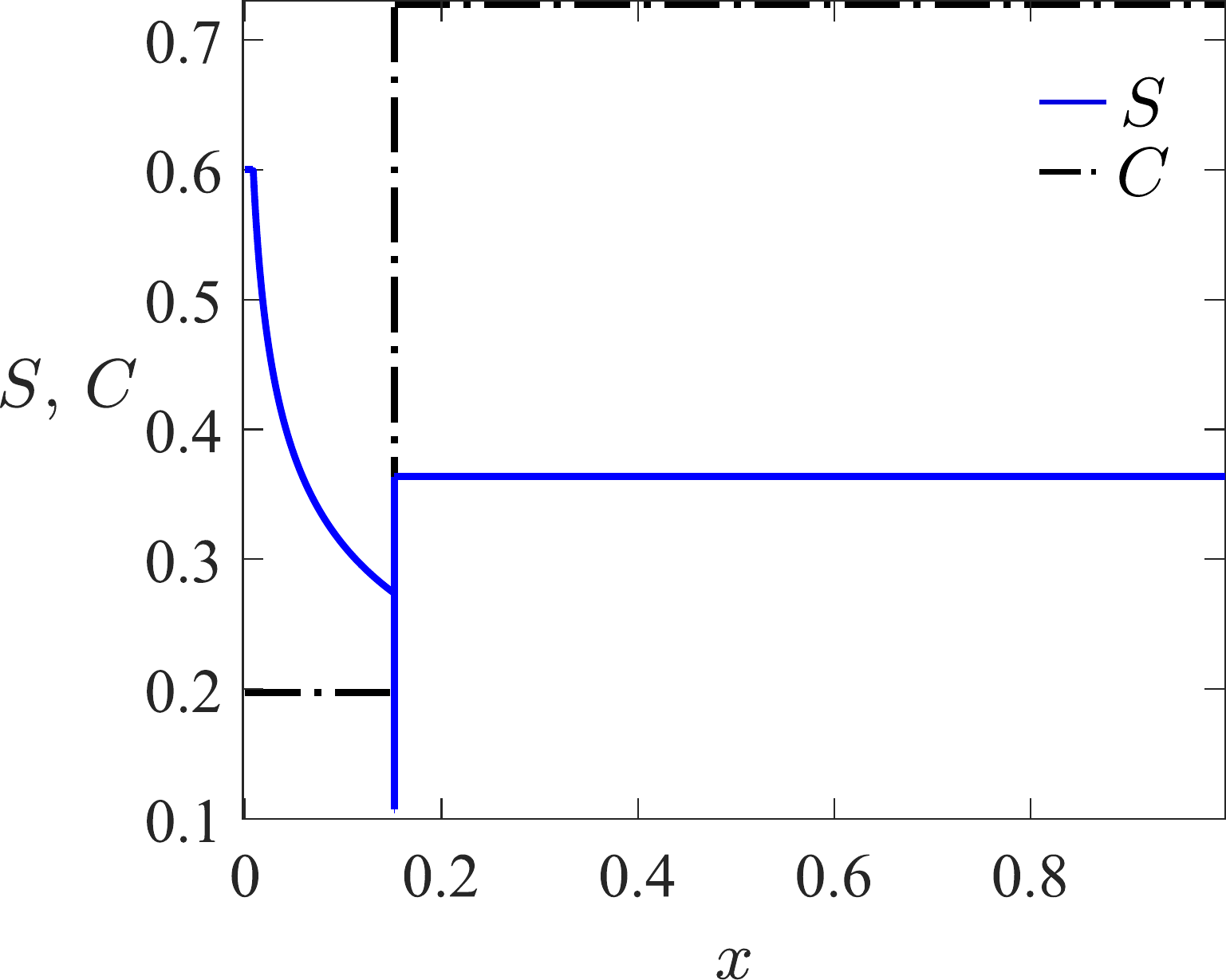}\hfill
	\includegraphics[height=5cm]{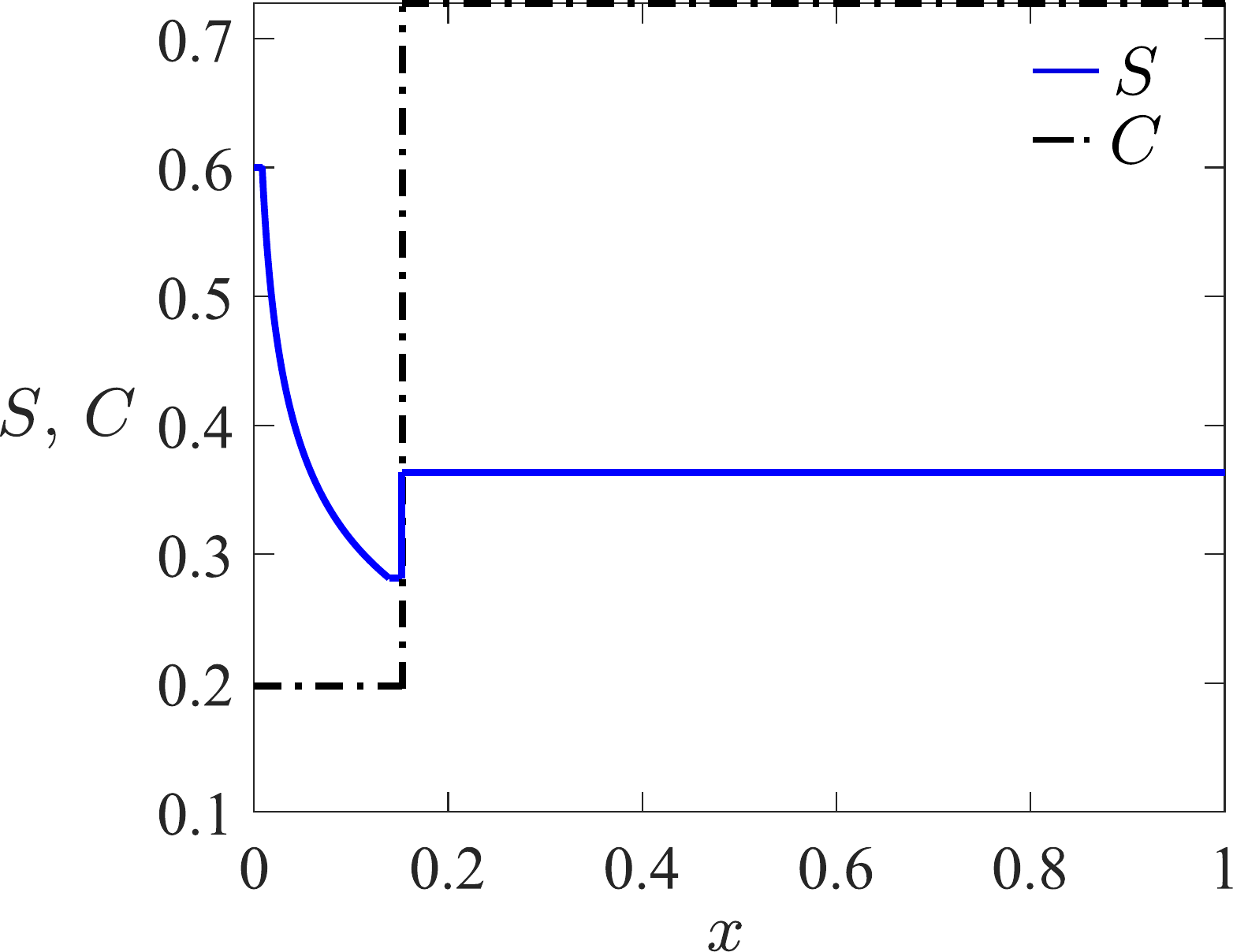}
	\caption{Solution profile at the same point in $\mathcal{R}_1 \cap \mathcal{R}_3$. The left panel shows the solution described in Lemma~\ref{Lemma-D1} and the right panel shows the one described in Lemma~\ref{Lemma-2}. This plot was done for the model explained in Section \ref{sec:application}, $U_L=(0.6, 0.1975540160002906)$, $U_R = (0.3636, 0.7273)$, time $t=0.5$ and the parameter values from Table \ref{Table:Values_parameter}.}
	\label{fig:Unicidade_R14}
\end{figure}

For the cases presented in Subsections \ref{sec:L1L2}, \ref{sec:L1L3}, and \ref{sec:R1R3}, the lack of uniqueness is related to sequences involving shocks and contact waves. In the $x$-$t$ (characteristic) plane, Fig.~\ref{Fig:Carac_Unicidade} presents the solutions corresponding to Figs.~\ref{fig:Unicidade_31}, \ref{fig:Unicidade_11}, and \ref{fig:Unicidade_R14}. 
As shock and contact wave velocities are equal in each of these cases, the solution is unique in the $x$-$t$ plane. The latter proves the following Lemma~\ref{Lemma:Unic_xt}.

\begin{figure}[h!]
	\centering	    
	\includegraphics[height=3.8cm]{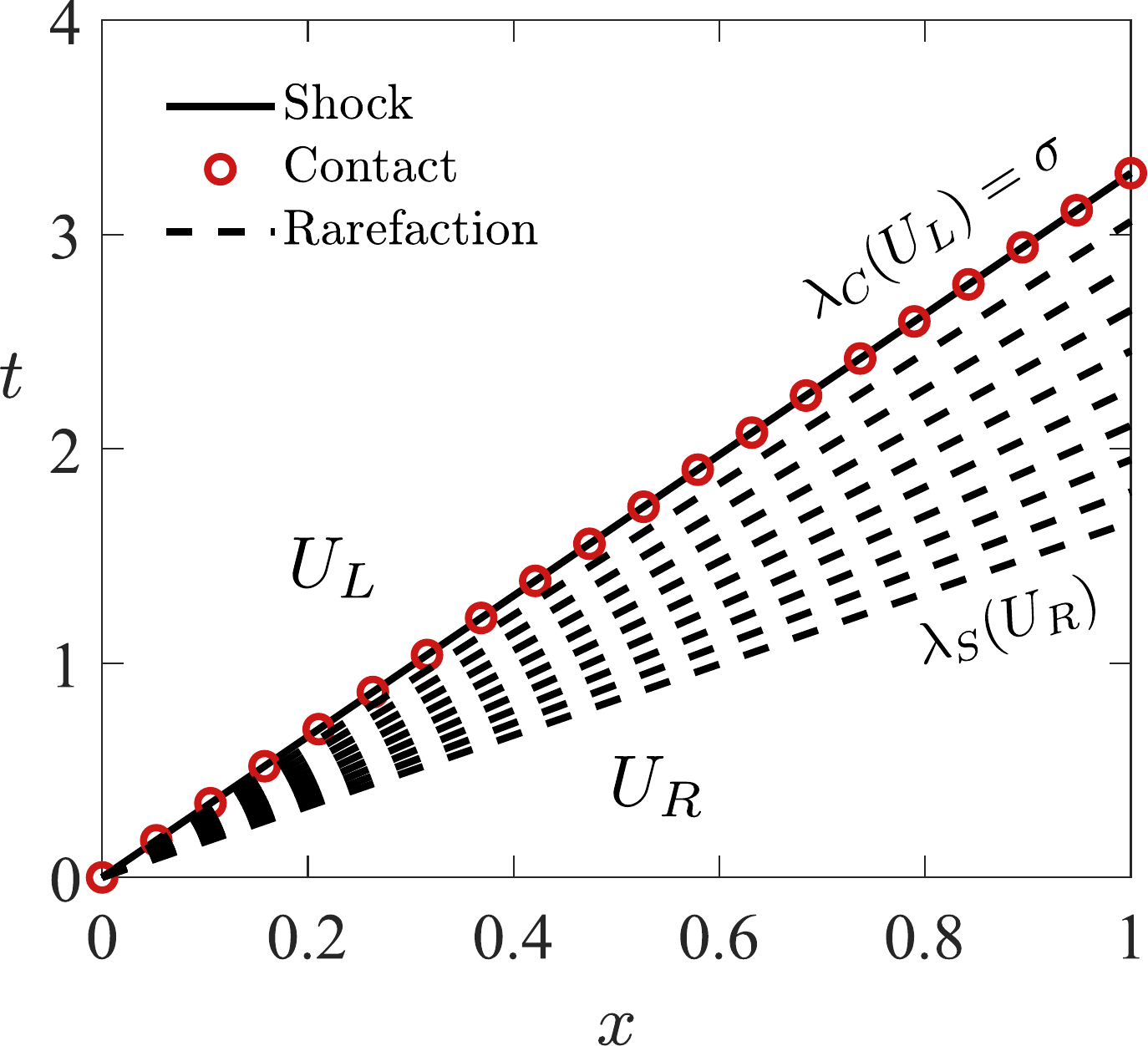}
	\includegraphics[height=3.8cm]{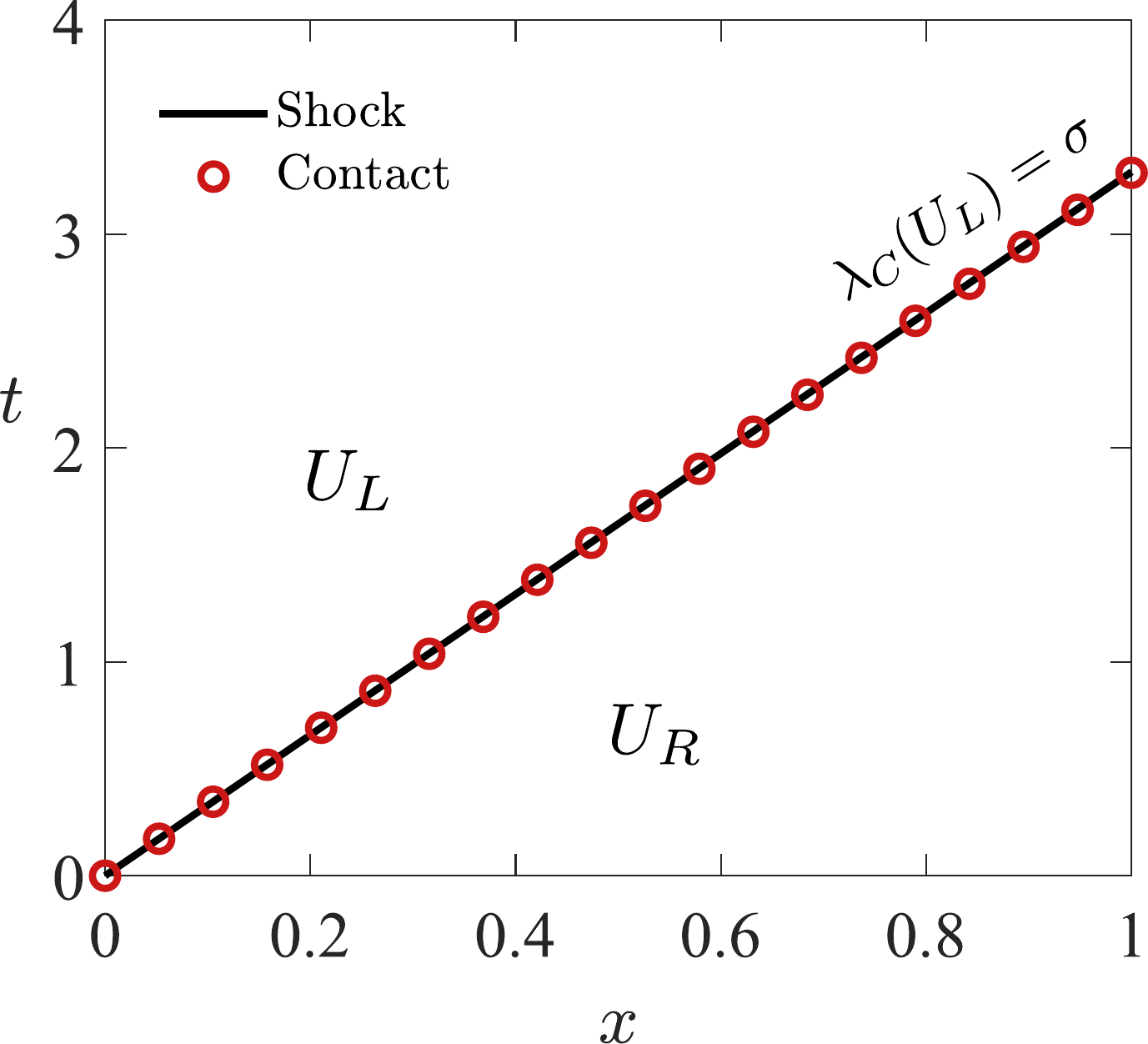}
	\includegraphics[height=3.8cm]{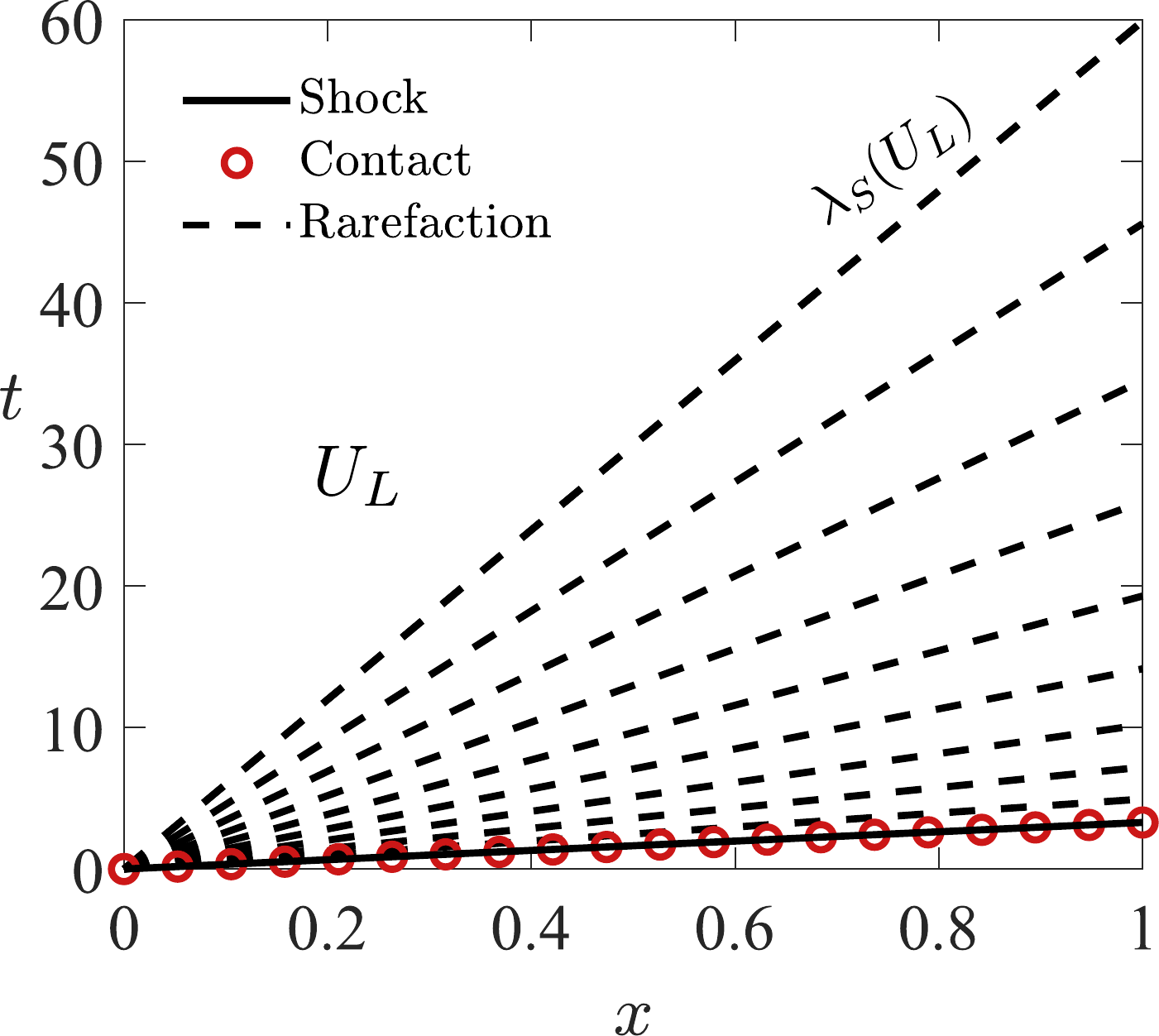}
 	\caption{Solutions for the cases presented in sub-sections \ref{sec:L1L2} (left panel), \ref{sec:L1L3} (central panel), and \ref{sec:R1R3} (right panel) in the $x$-$t$ plane.
  }
	\label{Fig:Carac_Unicidade}
\end{figure}

\begin{lemma}\label{Lemma:Unic_xt}
For a given left and right states ($U_L$ and $U_R$) in $I \times I$, there is a solution of the system \eqref{eq:mass_con_S}-\eqref{eq:evol_con_C} with initial data \eqref{eq:initial_cond}. This solution is unique for all $U_L$ and $U_R$ in $I \times I$, except on the set intersections $\mathcal{L}_1\cap \mathcal{L}_2$, $\mathcal{L}_1\cap \mathcal{L}_3$, and $\mathcal{R}_1\cap \mathcal{R}_3$ (defined in \eqref{eq:L1}-\eqref{eq:L3}, \eqref{eq:R1}-\eqref{eq:R3}), where we have two different solutions in the phase plane.
In any case, the solution is unique in the $x$-$t$ plane.
\end{lemma}

Let us consider the state $U_R$ in the intersections defined in the lemma above. Any neighborhood of this state possesses points in different sets, resulting in qualitatively different solutions. Thereby, small perturbation of $U_R$ results in a different sequence and, in some cases, a different number of waves. In this context, we enunciate the definition of structural stability \cite{thesisFred, schecter1996structurally}:
\begin{definition}
    The solution of the Riemann problem is said \textbf{structurally stable}, if the number and type of waves that compose the solution are preserved, when the initial data and the flux function are perturbed.
\end{definition}
The results proved in this section yield:
\begin{theorem}
    The Riemann problem \eqref{eq:mass_con_S}-\eqref{eq:evol_con_C}, and \eqref{eq:initial_cond} presents a loss of structural stability when initial conditions are in the sets $\mathcal{L}_1\cap \mathcal{L}_2$, $\mathcal{L}_1\cap \mathcal{L}_3$, and $\mathcal{R}_1\cap \mathcal{R}_3$.
\end{theorem}

Notice that, the lack of uniqueness happens one point, where the wave velocities of both solutions coincide. Despite the perturbations in $U_R$ result in different solutions, the waves velocity variation is continuous due to characteristics behavior in the $x$-$t$ plane. In addition, the variation in the intermediates states is bounded. Therefore, the solution $L^p$ norm for perturbed $U_R$, for a fixed time, depends continuously on initial conditions. As the same condition is valid for the points not at the intersections, the following result is valid:
\begin{theorem}
    The Riemann problem \eqref{eq:mass_con_S}-\eqref{eq:evol_con_C}, and \eqref{eq:initial_cond} is well-posed in $L^p$ norm, $p \in \mathbb{N}$.
\end{theorem}

\section{Application to foam displacement in porous media}
\label{sec:application}

In what follows, we describe the model implemented in the CMG-STARS simulator. Then, we show that the fractional flow function of this model satisfies the properties presented in Section~\ref{sec:RP}. Finally, we compare the analytical solution obtained in Section~\ref{sec:RP_construction} with the direct numerical simulations.

\subsection{CMG-STARS foam displacement model}
\label{App:A}

Let us consider the governing equations for an immiscible two-phase (gas-water), two-component displacement in one-dimensional flow through a porous medium assuming
Newtonian viscosity, no capillary pressure gradient, negligible gravitational effects, no physical dispersion, incompressible fluids, and a local foam steady-state (similar to \cite{Thorat2016foam, Zhang2009})
\begin{eqnarray}
\label{eq:mass_con}   \phi \partial_t S_w + \partial_x u_w &=& 0,\\
\label{eq:evol_C}   \partial_t\left[\phi \rho_w S_w C_s^w + (1-\phi)\rho_s C_s^s\right] + \partial_x\left[\rho_w C_s^w u_w\right] &=& 0,
\end{eqnarray}
where $\phi$ is the porosity, 
$S_w$ is the water saturation, 
$u_w$ is the water superficial velocity,
$\rho_w$ is the fluid density, $\rho_s$ is the rock density,
$C_s^w$ is the concentration of surfactant in the water phase, 
$C_s^s= K_d^a C_s^w$ is the surfactant adsorbed in the solid phase following Henry's adsorption, where $K_d^a$ is the {\it Freundlich} coefficient.
We consider a fully saturated porous medium $S_w+S_g=1$, where $S_g$ is the gas saturation. 
Equations \eqref{eq:mass_con} and \eqref{eq:evol_C} represent the water mass and total surfactant conservation, respectively.

The superficial velocities are given by Darcy's Law
\begin{equation}\label{eq:Darcy}
u_w = -(k k_{rw}/\mu_w )\nabla p, \quad u_g = -(k k^f_{rg}/\mu_g) \nabla p,
\end{equation}
where $k$ is the reference permeability, $\mu_j$ is the dynamic viscosity of phase $j$, $\nabla p$ is the pressure gradient (which is assumed constant), $k_{rj}$  is the $j$-phase relative permeability ($j=w, \ g$, where $w=$ water and $g=$ gas) is given by Brooks-Corey relation \cite{Brooks1966}
\begin{equation}\label{eq:relative_perm}
k_{rj} = k^0_{rj} \left(\dfrac{S_j - S_{jr}}{1-S_{wc}-S_{gr}} \right)^{n_j}, \quad j= w, g,
\end{equation}
where $k^0_{rj}$ is the end-point relative permeability, $n_j$ is Corey's exponent related to the wettability, $S_{wc}$ and $S_{gr}$ are the residual saturations of each phase.

In \eqref{eq:Darcy}, $k^f_{rg}$ represents the gas relative permeability modified by the water saturation and the effect of the surfactant concentration
\begin{equation}\label{eq:perm_foam}
k^f_{rg} = k_{rg} \cdot FM,
\end{equation}
where $FM$ is the mobility reduction factor \cite{Stars,Zhang2009} and is defined as
\begin{equation}\label{eq:FM}
FM(S_w,C_s^w)= \left(1+f_{mmob} \, F_1(C_s^w) \cdot F_2(S_w)\right)^{-1}.
\end{equation}
Here, $f_{mmob}$ indicates the reference mobility reduction factor \cite{Hematpur2018}; if this parameter is zero, there is no foam. The {\it dry-out} function $F_2$ depends on the water saturation:
\begin{equation}\label{eq:F2}
F_2(S_w) = 1/2 + {\arctan\left( epdry(S_w-fmdry)\right)}/{\pi},
\end{equation}
where $fmdry$ is a critical water saturation and $epdry$ indicates the abruptness of the dry-out effect, \cite{Zhang2009, Hematpur2018}. The function $F_1$ describes the surfactant concentration effect in the wetting phase:
\begin{align}
F_1(C_s^w) & =\begin{cases}
\left(\dfrac{C_s^w}{fmsurf}\right)^{epsurf}, & \text{if } C_s^w < fmsurf, \\
\quad \quad  1, & \text{if } C_s^w \geq fmsurf,\end{cases}
\end{align}
where {\it fmsurf} is the critical surfactant concentration and {\it epsurf} is an exponent parameter.

Now we rewrite the system \eqref{eq:mass_con} and \eqref{eq:evol_C} in the form of \eqref{eq:mass_con_S}-\eqref{eq:evol_con_C}. Let us now consider dimensionless time and length variables
\begin{equation}\label{eq:dim_var}
t' = (u\,t)/\left((1-S_{wc}-S_{gr}) \phi L\right) , \quad x' = x/L,
\end{equation}
where $L$ is the 1D length of the reservoir. The normalized saturation and surfactant concentration of each phase are
\begin{equation}\label{eq:norm_S}
S =(S_{w} - S_{wc})/\left(1-S_{wc}-S_{gr}\right), \quad C=C_s^w/C_{\text{max}}.
\end{equation}
Using \eqref{eq:dim_var} and \eqref{eq:norm_S}, the fractional flow function defined as $f=u_w/u$, where $u=u_g+u_w$, becomes
\begin{equation}\label{eq:frac_flow_norm}
f(S,C) = \dfrac{k_{rw}(S)}{k_{rw}(S) + (\mu_w / \mu_g) k^f_{rg}(S,C)}.
\end{equation}
Rewriting \eqref{eq:mass_con}-\eqref{eq:evol_C} in variables \eqref{eq:dim_var} and \eqref{eq:norm_S} yield the system \eqref{eq:mass_con_S}-\eqref{eq:evol_con_C} with $f$ defined in \eqref{eq:frac_flow_norm} and constant $\mathcal{A}$ given by
\begin{align}\label{constant_A}
\mathcal{A} = \left( S_{wc} + (1-\phi)(\rho_s/\rho_w \phi )K_d^a \right)/(1-S_{wc}-S_{gr}).
\end{align}

\subsection{Verification of the main properties of the fractional flow function}
\label{subsec:STARS_model}

In this Subsection, we show that the fractional flow function of water phase $f$ defined in Eq.~\eqref{eq:frac_flow_norm} satisfies conditions $a), \ b)$ and $c)$ presented at the beginning of Section~\ref{sec:RP}. As mentioned in the introduction, the difference between condition $c)$ and the corresponding one used in \cite{ELI,johansen1988} is related to the modeled physical phenomenon. In these works, the presence of the chemical tracer increases the viscosity of the wetting phase corresponding to the polymer flow. In the present study, the tracer reduces the mobility of the non-wetting phase, corresponding to foam flow in porous media \cite{de2021microscale}. 

\begin{itemize}
\item [$a$)] Due to the definition of the functions $k_{rw}, \ k_{rg}$ and $FM$ (see Subsection~\eqref{App:A}), we conclude that $f \in \mathscr{C}^2$. In addition, since $k_{rw}(0) = 0$, and applying the limit on $f$ we conclude that $f(0, C)=0$ for every $C \in I$. Notice that, $k_{rg}^f(1,C)=0$, therefore $f(1,C)=1$ for every $C \in I$.

The partial derivative of $f$ in $S$ is given by
\begin{eqnarray}\label{eq:STARS_f_S}
\partial_S f(S,C) = \dfrac{\mu_w}{\mu_g} \left(\frac{d_S k_{rw}(S)k_{rg}^f(S,C) - k_{rw}(S)\partial_S k_{rg}^f(S,C)}{k_{rw}(S)+(\mu_w/\mu_g) k_{rg}^f(S,C)}\right).
\end{eqnarray}
Once $k_{rw}(0)=0=d_S k_{rw}(0)$ and $k_{rg}^f(1,C)=0=\partial_S k_{rg}^f(1,C)$ for every $C \in I$, we conclude that $\partial_S f(0,C)=0=\partial_Sf(1,C)$. 

\item[$b$)] Notice that $d_S k_{rw}(S), \ k_{rg}^f(S,C), \ k_{rw}(S)$ are positive and $\partial_S k_{rg}^f(S,C)$ is negative for $C \in I$ and $S \in (0,1)$. Using \eqref{eq:STARS_f_S} yields $\partial_S f(S,C) > 0$ for $C \in I$ and $S \in (0,1)$.

In Fig.~\ref{fig:Inflexion}, we provide a numerical approximation of the second derivative of the fractional flow function in $S$ for several values of $C$. To allow this analysis we use quadratic Corey's exponents, which are consistent with laboratory results \cite{Marchesin1992}. For values reported in Table \ref{Table:Values_parameter}, a single inflection point is observed for each $C$. However, rigorous proof remains an open problem.
\begin{figure}[h!]
\centering	
\includegraphics[height=5cm]{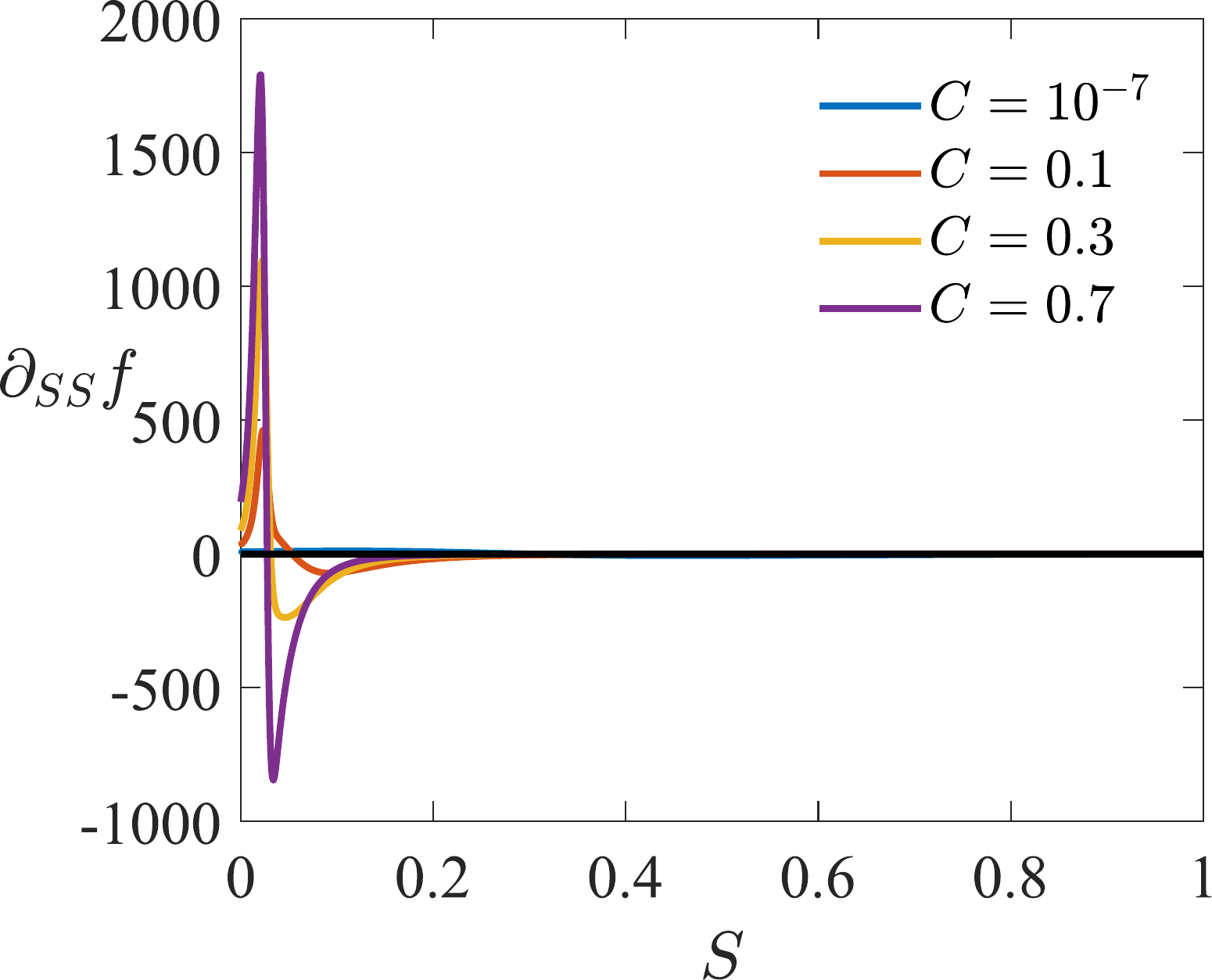}
\includegraphics[height=5cm]{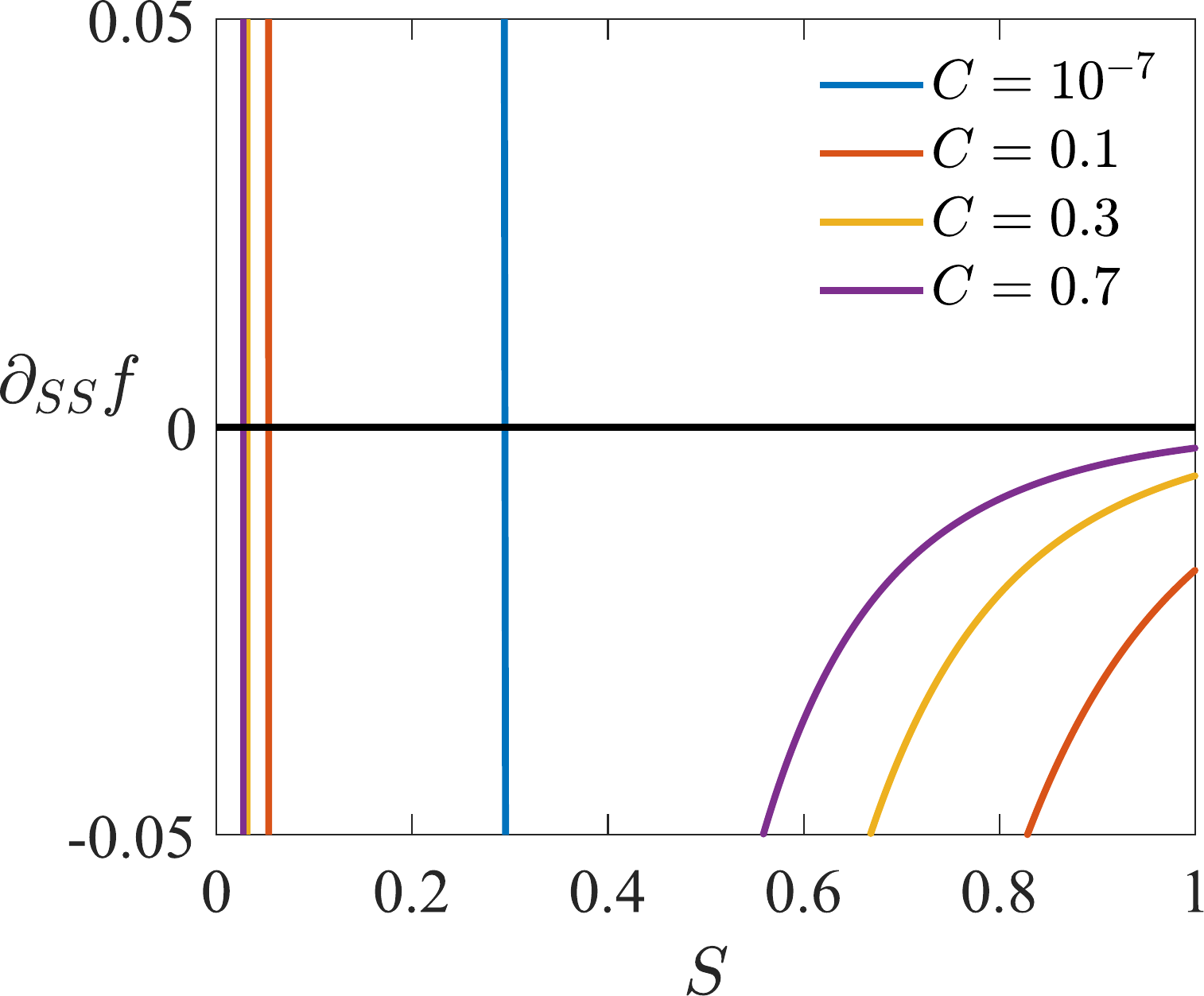}
\caption{The second derivative of $f$ as function of $C$. The right panel is a zoom of the left panel for $\partial_{SS}f$ between $-0.05$ and $0.05$.}
\label{fig:Inflexion}
\end{figure}

\item [$c$)] The derivative of $f$ in relation to $C$ is given by
\begin{eqnarray}
\partial_C f(S,C) = - \dfrac{\mu_w}{\mu_g} f(S,C)\left( \frac{\partial_C k_{rg}^f(S,C)}{k_{rw}(S)+{\mu_w}/{\mu_g} k_{rg}^f(S,C)}\right).
\end{eqnarray}
Since $fmsurf=C_{max}$ and $epsurf=1$, we obtain $\partial_C k_{rg}^f(S,C) < 0$ for all $C \in I$ and $S \in (0,1)$. As $k_{rw}(S) >0$, $f(S,C)>0$ and $ k_{rg}^f(S,C)>0$ for all $S \in (0,1)$ and $C \in I$, we conclude that $\partial_C f(S,C) > 0$ for all $C \in I$ and $S\in (0,1)$.
\end{itemize}
	
\subsection{Numerical simulations}
\label{subsec:numerical}

The solution of the system \eqref{eq:mass_con_S}-\eqref{eq:evol_con_C} together with initial condition \eqref{eq:initial_cond} is solved using the RCD solver \cite{lambert2020mathematics}. This second-order solver is based on the implicit finite differences Crank-Nicolson scheme combined with Newton's method. The boundary conditions were: Dirichlet and no-flow Neumann on the left and right sides, respectively. For the space discretization, we used 7000 points, and for the time discretization, the time-step was $10^{-6}$. We used parameter values summarized in Table~\ref{Table:Values_parameter}.

\begin{table}[ht]
\centering
\caption{The parameter values used in numerical simulations. Source: \cite{Stars,Valdez2021,Marchesin1992}.}
\begin{tabular}{l|l|l}
	Symbol & Parameter & Value  \\
	\hline
	$k^0_{rw}$ & End-point water relative permeability & 0.302\\
	$k^0_{rg}$ & End-point gas relative permeability& 0.004\\
	$n_w$ & Corey's exponent for water & 2 \\ 
	$n_g$ & Corey's exponent for gas & 2 \\  
	$\mu_w$ & Water viscosity &1e-03 [Pa $\cdot$ s] \\
	$\mu_g$ & Gas viscosity & 5e-05 [Pa $\cdot$ s] \\ 
	$\phi$ & Porosity & 0.21 \\
	$\rho_w$ & Water density & 1000\\ 
	$\rho_s$ & Solid density & 2000\\ 
	$fmmob$ & Mobility reduction factor & 293.27 \\
	$fmdry$ & Critical water saturation & 0.437 \\
	$epdry$ & Abruptness of dry out effect & 359.33 \\
	$S_{wc}$ & Water connate saturation & 0.43 \\ 
	$S_{gr}$ & Gas residual saturation & 0.293 \\
	$fmsurf$ & Critical surfactant concentration &2 [g/L] \\ 
	$epsurf$ & Foam strength coefficient &1 [g/L]\\ 
	$C_{max}$ & Maximum surfactant concentration & 2 \\
	$K_d^a$ & Adsorption constant & 0.05\\
	\hline
\end{tabular}
\label{Table:Values_parameter}
\end{table}

We provide two examples of the solution profile for water saturation and surfactant concentration in Fig.~\ref{Fig:numerical_example_1} and compare them to numerical simulations. 
The left panel in Fig.~\ref{Fig:numerical_example_1} shows the solution construction for a drainage case (injection of gas into the reservoir filled with water) composed of a contact wave followed by a shock wave. The solution, following Lemma~\ref{Lemma-1} is composed of a $C$-wave followed by a shock $S$-wave. 
The right panel in Fig.~\ref{Fig:numerical_example_1} shows the solution construction for imbibition (injection of water into the reservoir filled with gas). This case corresponds to Lemma \ref{Lemma-D2} with the solution given by an $S$-wave (in this case, a rarefaction wave and a shock wave) followed by a contact wave and another $S$-wave. In both cases, the analytical and numerical solutions present an excellent agreement.
\begin{figure}[h!]
\centering	    
\includegraphics[height=5cm]{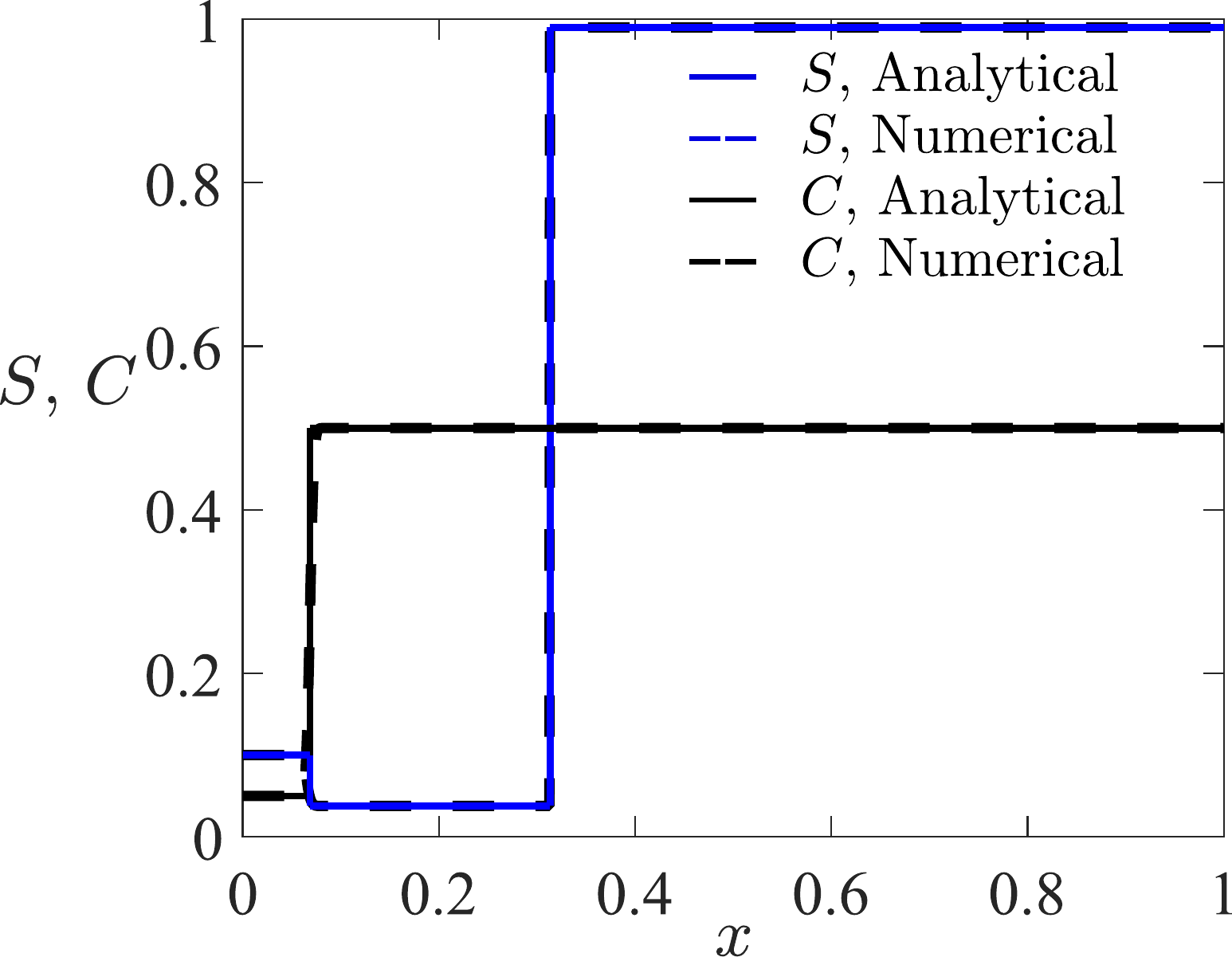} \hfill
\includegraphics[height=5cm]{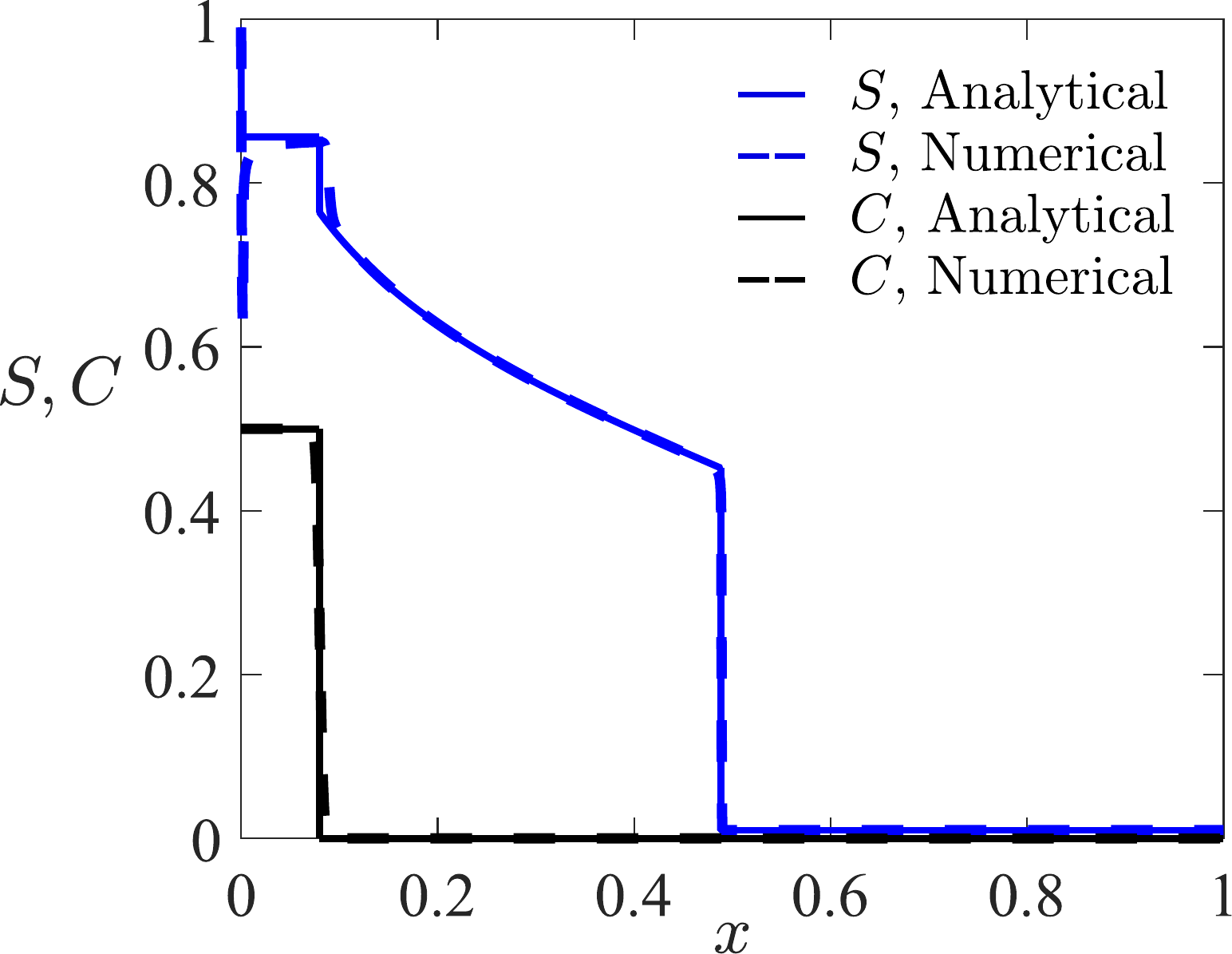}
\caption{Comparison between the numerical (dashed lines) and analytical (solid lines) solutions of the water saturation $S$ (blue) and surfactant concentration $C$ (black) profiles using data from Table~\ref{Table:Values_parameter}. The left panel corresponds to $U_L=(0.1, 0.05)$, $U_R=(0.99,0.5)$, and time $t = 0.5$. The right panel corresponds to $U_L=(0.99, 0.5)$, $U_R=(0.01,0)$, and $t = 0.3$. } 
\label{Fig:numerical_example_1}
\end{figure}

Next, we show the numerical effect resulting from the lack of structural stability presented in Section~\ref{Sec:Uniqueness}. Figure \ref{Fig:Unicidade_2_RCD} presents the water saturation profiles for $U_R \in \mathcal{L}_1 \cap \mathcal{L}_3$ (corresponding to Fig.~\ref{fig:Unicidade_PF} in the phase plane and to solution profiles in Fig.~\ref{fig:Unicidade_11}).
From the numerical perspective, it is impossible to pick the point exactly at the intersection of two sets. That is why, we choose $S_R$ values in the neighborhood of the intersection: $U_{R+}=(0.367,0.7273)$ and $U_{R-}=(0.360,0.7273)$. Notice that, for perturbations in $S_R$ the water saturation presented qualitatively different profiles. As previously commented, this behavior is not related to numerical issues, but the loss of structural stability by the model in the neighborhood. Cases for $U_R \in \mathcal{L}_1 \cap \mathcal{L}_2$ and $U_R \in \mathcal{R}_1 \cap \mathcal{R}_3$ are analogous.
\begin{figure}[h!]
\centering	    
\includegraphics[height=5cm]{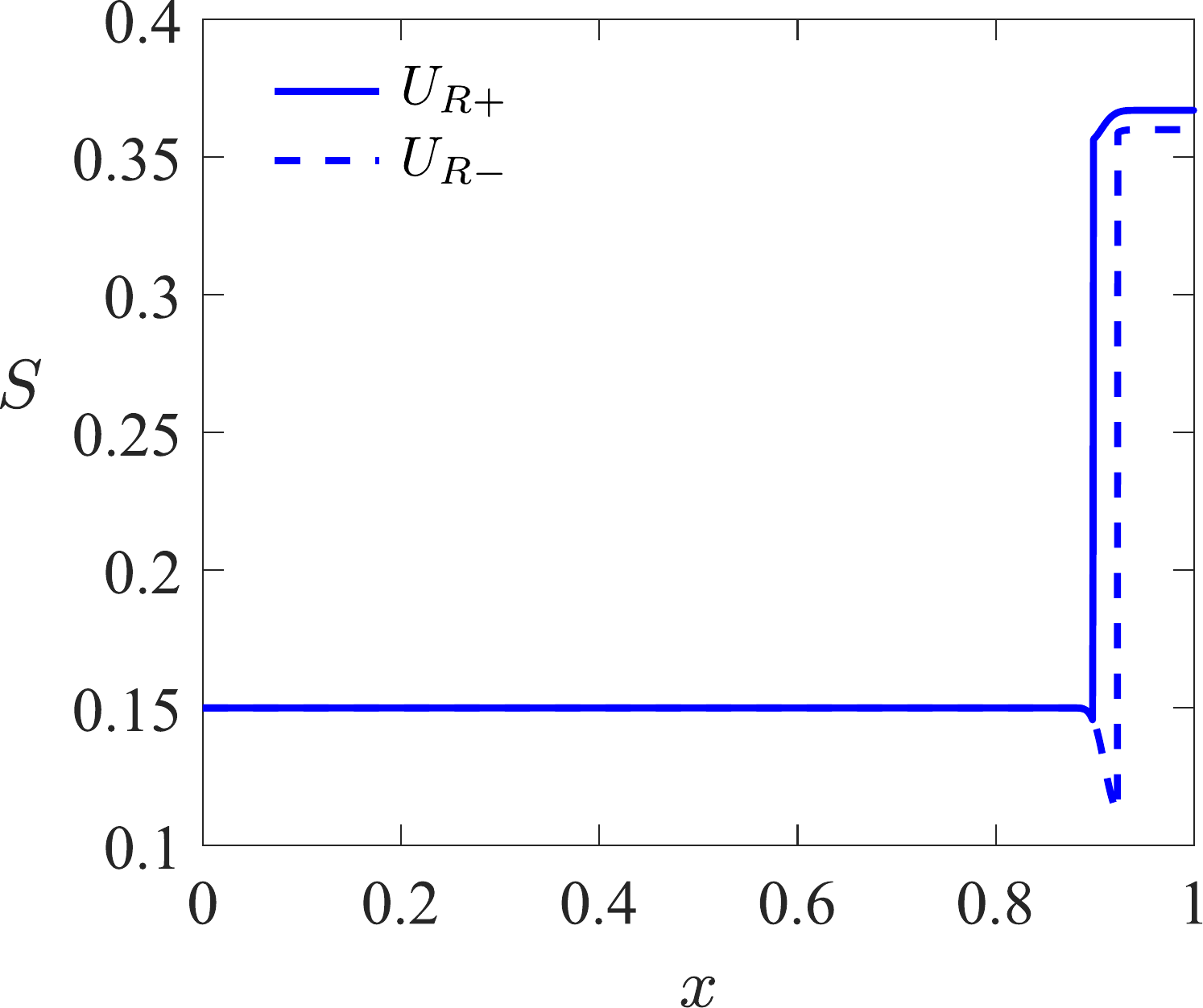} 
\caption{Water saturation profiles at time $t=3$ for the left state $U_L=(0.15,0.4)$ and two nearby right states $U_{R-}=(0.360,0.7273)$ and $U_{R+}=(0.367,0.7273)$ in the neighborhood of the intersection $U_R \in \mathcal{L}_1 \cap \mathcal{L}_3$.} 
\label{Fig:Unicidade_2_RCD}
\end{figure}

\section{Conclusions}
\label{sec:conclusions}

Motivated by the foam displacement in porous media with linear adsorption, we extended the existing framework for the two-phase flow containing an active tracer described by a non-strictly hyperbolic system of conservation laws.
We solved the global Riemann problem by presenting possible wave sequences that composed this solution. Although the problem is well-posed for all Riemann data, we identified parameter regions where the solution lacks structural stability.

We verified that the CMG-STARS model describing foam displacement in porous media with linear adsorption satisfies the hypotheses to apply the developed theory. Therefore, there exists a parameter region where the CMG-STARS model losses the structural stability, which can result in numerical oscillations. Our results provide a possible explanation to some numerical issues appearing in commercial simulators.


\section{Acknowledgments}

We thank Prof. Dr. F. Furtado, Prof. Dr. A. Pires and Prof. Dr. Y. Petrova for helpful discussions improving the mathematical quality of this work.

\bibliographystyle{siamplain}
\bibliography{references}
\end{document}